\documentclass[11pt]{article}

\usepackage[margin=1.2in]{geometry}

\usepackage[T1]{fontenc}
\usepackage[utf8]{inputenc}
\usepackage[USenglish, UKenglish, english]{babel}
\usepackage{amssymb}
\usepackage{amsthm}
\usepackage{amsmath}
\usepackage[mathscr]{euscript}
\usepackage{enumitem}
\usepackage{graphicx}
\usepackage{MnSymbol}
 \let\mathscr\relax
\usepackage[scr]{rsfso}
\usepackage{tikz-cd}
\usepackage{tikz}
\usetikzlibrary{matrix,arrows,decorations.pathmorphing}
\usepackage{hyperref}
\usepackage{marginnote}
\usetikzlibrary{arrows.meta}

\usepackage{cases}

\usepackage{libertine}

\theoremstyle{definition}
\newtheorem{defin}{Definition}[section]
\theoremstyle{definition}

\theoremstyle{plain}
\newtheorem{theo}[defin]{Theorem}
\theoremstyle{plain}
\newtheorem{prop}[defin]{Proposition}
\theoremstyle{plain}
\newtheorem{lem}[defin]{Lemma}
\theoremstyle{plain}
\newtheorem{cor}[defin]{Corollary}
\theoremstyle{definition}
\newtheorem{rmk}[defin]{Remark}
\theoremstyle{definition}

\theoremstyle{definition}

\theoremstyle{plain}

\theoremstyle{definition}

\theoremstyle{definition}

\theoremstyle{definition}

\theoremstyle{plain}

\theoremstyle{definition}

\theoremstyle{definition}
\newtheorem*{defin*}{Definition}
\theoremstyle{definition}
\newtheorem*{ex*}{Example}
\theoremstyle{plain}
\newtheorem*{theo*}{Theorem}
\theoremstyle{plain}
\theoremstyle{plain}
\newtheorem*{conj*}{Conjecture}
\newtheorem*{prop*}{Proposition}
\theoremstyle{plain}
\newtheorem*{lem*}{Lemma}
\theoremstyle{plain}
\newtheorem*{cor*}{Corollary}
\theoremstyle{definition}
\newtheorem*{rmk*}{Remark}
\theoremstyle{definition}
\newtheorem*{exe*}{Exercise}

\theoremstyle{plain}
\newtheorem{theoA}{Theorem}

\theoremstyle{plain}

\theoremstyle{plain}

\theoremstyle{plain}

\theoremstyle{plain}

\numberwithin{equation}{section}
\usepackage{parskip}

\makeatletter
\def\thm@space@setup{%
  \thm@preskip=\parskip \thm@postskip=0pt
}
\makeatother

\setlist[enumerate]{label=(\roman*)}
\def\R{{\mathbf{R}}} %Twisted/parabolic induction

\def\uch{{\rm Uch}}

\def\bl{{\rm bl}}

\def\irr{{\rm Irr}}

\def\ps{{\rm ps}}

\def\ker{{\rm Ker}}

\def\aut{{\rm Aut}}

\def\cl{{\mathfrak{Cl}}}

\def\GL{{\rm GL}}

\def\n{{\mathbf{N}}} %normalizer

\def\c{{\mathbf{C}}} %centralizer
 
\def\z{{\mathbf{Z}}} %center

\def\O{{\mathbf{O}}} %radical

\def\E{{\mathcal{E}}} %Lusztig series

\def\C{{\mathcal{C}}} %couples chain-character

\def\Pr{{\mathcal{P}}} %Projective representation

\def\k{{\mathbf{k}}}

\def\G{{\mathbf{G}}} %algebraic group

\def\H{{\mathbf{H}}} %algebraic subgroup

\def\K{{\mathbf{K}}} %algebraic subgroup

\def\L{{\mathbf{L}}} %Levi

\def\M{{\mathbf{M}}} %Levi

\def\T{{\mathbf{T}}} %Torus

\def\S{{\mathbf{S}}} %Torus

\def\B{{\mathbf{B}}} %Borel

\def\P{{\mathbf{P}}} %Parabolic

\def\CL{{\mathcal{L}}} %couples (chain of Levis)-character

\DeclareMathOperator{\CP}{\mathcal{C}\CPkern \mathcal{P}_{\rm u}}
\newcommand{\CPkern}{%
  \mkern-1.0mu
  \mathchoice{}{}{\mkern0.2mu}{\mkern0.5mu}%
}
\makeatletter
\newcommand{\uset}[3][0ex]{%
  \mathrel{\mathop{#3}\limits_{
    \vbox to#1{\kern-7\ex@
    \hbox{$\scriptstyle#2$}\vss}}}}
\makeatother

\newcommand{\wt}[1]{\widetilde{#1}} 

\newcommand{\wh}[1]{\widehat{#1}}

\newcommand{\ws}[1][1.5]{
  \mathrel{\scalebox{#1}[1]{$\sim$}}
}

\newcommand{\iso}[1]{\ws_{#1}}

\newcommand{\isoc}[1]{\ws_{#1}^c}

\usepackage{sectsty}
\allsectionsfont{\centering}

\usepackage{xparse,etoolbox}

\newcommand{\blocktheorem}[1]{%
  \csletcs{old#1}{#1}% Store \begin
  \csletcs{endold#1}{end#1}% Store \end
  \RenewDocumentEnvironment{#1}{o}
    {\par\addvspace{1.5ex}
     \noindent\begin{minipage}{\textwidth}
     \IfNoValueTF{##1}
       {\csuse{old#1}}
       {\csuse{old#1}[##1]}}
    {\csuse{endold#1}
     \end{minipage}
     \par\addvspace{1.5ex}}
}

\blocktheorem{theo}
\blocktheorem{conj}
\blocktheorem{para}
\blocktheorem{theoA}
\blocktheorem{conjA}
\blocktheorem{paraA}
\blocktheorem{corA}

\makeatletter
\def\blfootnote{\gdef\@thefnmark{}\@footnotetext}
\makeatother

\title{
{\huge\bf A local-global principle for unipotent characters}\\
\author{\Large Damiano Rossi}
\date{}
\blfootnote{\emph{$2010$ Mathematical Subject Classification:} $20$C$20$, $20$C$33$.
\\
\emph{Key words and phrases:} Dade's Conjecture, Character Triple Conjecture, finite reductive groups, unipotent characters.
\\
This work is partially supported by the EPSRC grant EP/T$004592/1$ and was written during a research visit of the author at the Universit\'a degli Studi di Firenze. The author would like to thank Silvio Dolfi and all the members of the algebra group in the Department of Mathematics for their hospitality and, in particular, Carolina Vallejo for some comments concerning the local-global principle. Moreover, the author would like to thank Lucas Ruhstorfer for some helpful conversation on the paper \cite{Bro-Ruh23}.
}}

\begin{document}

\renewcommand{\thetheoA}{\Alph{theoA}}

\renewcommand{\thecorA}{\Alph{corA}}

\selectlanguage{english}

\maketitle

\begin{abstract}
We obtain an adaptation of Dade's Conjecture and Sp\"ath's Character Triple Conjecture to unipotent characters of simple, simply connected finite reductive groups of type $\bf{A}$, $\bf{B}$ and $\bf{C}$. In particular, this gives a precise formula for counting the number of unipotent characters of each defect $d$ in any Brauer $\ell$-block $B$ in terms of local invariants associated to $e$-local structures. This provides a geometric version of the local-global principle in representation theory of finite groups. A key ingredient in our proof is the construction of certain parametrisations of unipotent generalised Harish-Chandra series that are compatible with isomorphisms of character triples.
\end{abstract}

\tableofcontents

\section{Introduction}

The local-global conjectures are currently some of the most interesting and challenging problems in representation theory of finite groups. Among others, these include the McKay Conjecture \cite{McK72}, the Alperin–-McKay Conjecture \cite{Alp76} and Alperin’s Weight Conjecture \cite{Alp87} all of which can be deduced by a deeper statement known as Dade's Conjecture \cite{Dad92}, \cite{Dad94}, \cite{Dad97}. The latter also implies the celebrated Brauer's Height Zero Conjecture introduced in \cite{Bra55} and whose proof has recently been completed in \cite{MNSFT} and \cite{Ruh22AM} while relying on a combined effort of many other authors.

In this paper, we are particularly interested in \textit{Dade's Conjecture} which, for every prime number $\ell$, suggests a precise formula for counting the number of irreducible characters of a finite group, with a given $\ell$-defect and belonging to a given Brauer $\ell$-block, in terms of the $\ell$-local structure of the group itself. This conjecture has been further extended in \cite{Spa17} where the \textit{Character Triple Conjecture} was formulated by introducing a compatibility with \textit{$N$-block isomorphisms} of character triples, hereinafter denoted by $\iso{N}$, as defined in \cite[Definition 3.6]{Spa17}. This notion plays a fundamental role in many aspects of group representation theory and, as we will see later, gives us a way to control the representation theory of local subgroups. Furthermore, it was exploited to reduce Dade's Conjecture to finite quasi-simple groups as explained in \cite[Theorem 1.3]{Spa17}.

Our aim is to adapt and prove the two conjectures described in the previous paragraph to the case of unipotent characters of finite reductive groups. The approach considered here is inspired by ideas introduced by the author in \cite{Ros-Generalized_HC_theory_for_Dade} and provides further evidence for the conjectures formulated in that paper \cite[Conjecture C and Conjecture D]{Ros-Generalized_HC_theory_for_Dade}. In particular, the $\ell$-local structures considered above are replaced by more suitable $e$-local structures arising from the geometry of the underlying algebraic group that are compatible with the framework of Deligne--Lusztig theory. Therefore, our results also suggest the existence of an $e$-local-global principle for the representation theory of finite reductive groups.

More precisely, let $\G$ be a simple, simply connected group of type $\bf{A}$, $\bf{B}$ or $\bf{C}$ which is defined over an algebraically closed field of positive characteristic $p$ and let $F:\G\to\G$ be a Frobenius endomorphism endowing $\G$, as a variety, with an $\mathbb{F}_q$-structure for some power $q$ of $p$. We denote by $\G^F$ the \textit{finite reductive group} consisting of the $\mathbb{F}_q$-rational points on $\G$. Furthermore, we fix an odd prime $\ell$ different from $p$ and denote by $e$ the multiplicative order of $q$ modulo $\ell$. We let $\CL_e(\G,F)$ denote the set of \textit{$e$-chains} of $(\G,F)$ of the form $\sigma=\{\G=\L_0>\L_1>\dots>\L_n\}$ where each $\L_i$ is an $e$-split Levi subgroup of $(\G,F)$. The final term of the $e$-chain $\sigma$ is denoted by $\L(\sigma)=\L_n$, while $|\sigma|:=n$ is the length of $\sigma$. Observe that the latter induces a partition of the set $\CL_e(\G,F)$ into the sets $\CL_e(\G,F)_\pm$ consisting of those $e$-chains $\sigma$ that satisfy $(-1)^{|\sigma|}=\pm 1$. Furthermore, notice that $\G^F$ acts by conjugation on the set $\CL_e(\G,F)$ and indicate by $\G_\sigma^F$ the stabiliser of the $e$-chain $\sigma$. It follows directly from the definition that this action preserves the length of $e$-chains and, in particular, it restricts to an action of $\G^F$ on the set $\CL_e(\G,F)_{>0}$ of $e$-chains of positive length.

Now, to each non-negative integer $d$ and Brauer $\ell$-block $B$ of the finite group $\G^F$, we associate a set $\CL_{\rm u}^d(B)_\pm$ consisting of quadruples $(\sigma,\M,\mu,\vartheta)$ where $\sigma$ is an $e$-chain belonging to $\CL_e(\G,F)_\pm$, $(\M,\mu)$ is a unipotent $e$-cuspidal pair of $(\L(\sigma),F)$ such that $\M$ does not coincide with $\G$, and $\vartheta$ is an irreducible character of the $e$-chain stabiliser $\G_\sigma^F$ belonging to the character set $\uch^d(B_\sigma,(\M,\mu))$ defined by the choice of $d$, $B$, $\sigma$ and $(\M,\mu)$ as described in Definition \ref{def:e-chains character sets, blocks}. Once again, the group $\G^F$ acts by conjugation on $\CL_{\rm u}^d(B)_\pm$ and we indicate the corresponding set of $\G^F$-orbits by $\CL^d_{\rm u}(B)_\pm/\G^F$. Moreover, for every such orbit $\omega$, we denote by $\omega^\bullet$ the corresponding $\G^F$-orbit of pairs $(\sigma,\vartheta)$ such that $(\sigma,\M,\mu,\vartheta)\in\omega$ for some unipotent $e$-cuspidal pair $(\M,\mu)$.

With the above notation, we are now able to state our first main result. For simplicity, in the next theorem we assume that the prime $\ell$ does not divide the greatest common divisor $(q\pm 1,n+1)$ whenever $(\G,F)$ is of type ${\bf A}_n(\pm q)$ and where ${\bf A}_n(-q)$ denotes ${^2{\bf A}}_n(q)$ as usual. Observe however that this assumption can be removed as explained in Remark \ref{rmk:Condition on prime} (see Theorem \ref{thm:iUnipotent} for the more general statement).

\begin{theoA}
\label{thm:Main iUnipotent}
For every Brauer $\ell$-block $B$ of $\G^F$ and every non-negative integer $d$, there exists an $\aut_\mathbb{F}(\G^F)_B$-equivariant bijection
\[\Lambda:\CL^d_{\rm u}(B)_+/\G^F\to\CL^d_{\rm u}(B)_-/\G^F\]
such that
\[\left(X_{\sigma,\vartheta},\G^F_\sigma,\vartheta\right)\iso{\G^F}\left(X_{\rho,\chi},\G^F_\rho,\chi\right)\]
for every $\omega\in\CL^d_{\rm u}(B)_+/\G^F$, any $(\sigma,\vartheta)\in\omega^\bullet$, any $(\rho,\chi)\in\Lambda(\omega)^\bullet$ and where $X:=\G^F\rtimes \aut_\mathbb{F}(\G^F)$ and $\aut_\mathbb{F}(\G^F)$ is the group of automorphisms described in Section \ref{sec:equivarance and maximal ext}. 
\end{theoA}

The above theorem provides an adaptation of Sp\"ath's Character Triple Conjecture to the framework of Deligne--Lusztig theory for the unipotent characters of finite reductive groups. Theorem \ref{thm:Main iUnipotent} also offers further evidence for the validity of \cite[Conjecture D]{Ros-Generalized_HC_theory_for_Dade}, in fact the set $\CL_{\rm u}^d(B)_\pm$ introduced above is a subset of the set of quadruples $\CL^d(B)_\pm$ considered in \cite[Conjecture D]{Ros-Generalized_HC_theory_for_Dade} which is identified by only selecting unipotent $e$-cuspidal pairs $(\M,\mu)$ among those appearing in such quadruples.

Next, we obtain a formula for counting the number of unipotent characters of $\ell$-defect $d$ in the Brauer $\ell$-block $B$ in terms of local invariants associated to $e$-local structures. For each $e$-chain $\sigma$ of $(\G,F)$ with positive length, we define $\k^d_{\rm u}(B_\sigma)$ to be the number of characters belonging to one of the character sets $\uch^d(B_\sigma,(\M,\mu))$ for some unipotent $e$-cuspidal pair $(\M,\mu)$ of $(\L(\sigma),F)$ up to $\G^F_\sigma$-conjugation (see also \eqref{eq:number of local characters}). Furthermore, let $\k^d_{\rm u}(B)$ and $\k_{\rm c, u}^d(B)$ be the number of irreducible characters with $\ell$-defect $d$ and belonging to the Brauer $\ell$-block $B$ that are unipotent and unipotent $e$-cuspidal respectively. Then, by using the bijection given by Theorem \ref{thm:Main iUnipotent} we can determine the difference $\k^d_{\rm u}(B)-\k^d_{\rm c, u}(B)$ in terms of an alternating sum involving the terms $\k^d_{\rm u}(B_\sigma)$ arising from the $e$-local structure $\G_\sigma^F$.

\begin{theoA}
\label{thm:Main Unipotent}
For every Brauer $\ell$-block $B$ of $\G^F$ and every non-negative integer $d$, we have
\[\k^d_{\rm u}(B)-\k^d_{\rm c,u}(B)=\sum\limits_{\sigma}(-1)^{|\sigma|+1}\k_{\rm u}^d(B_\sigma)\]
where $\sigma$ runs over a set of representatives for the action of $\G^F$ on $\CL_e(\G,F)_{>0}$.
\end{theoA}

We point out that the restriction on the prime $\ell$ made for simplification before Theorem \ref{thm:Main iUnipotent} only concerns the condition on isomorphisms of character triples and hence does not affect Theorem \ref{thm:Main Unipotent}. As before, this result provides an adaptation of Dade's Conjecture to the framework of Deligne--Lusztig theory for the unipotent characters of finite reductive groups and gives new evidence in favour of \cite[Conjecture C]{Ros-Generalized_HC_theory_for_Dade}. The necessity for the introduction of the corrective term $\k^d_{\rm c,u}(B)$ in the equality of Theorem \ref{thm:Main Unipotent} can be understood as an analogue to the exclusion of the case of blocks with central defect in the statement of Dade's Conjecture or, depending on the formulation under consideration, of the case where $d=0$. We refer the reader to the more detailed discussion given in the paragraph following Definition \ref{def:e-chains character sets}. Finally, we mention that Theorem \ref{thm:Main Unipotent} also provides evidence for a positive answer to a question recently posed by Brou\'e \cite{Broue22}.

It is particularly interesting to notice that, to the author's knowledge, Theorem \ref{thm:Main Unipotent} cannot be obtained directly using techniques available at the present time, but only as a consequence of the existence of $\G^F$-block isomorphisms of character triples as those considered in Theorem \ref{thm:Main iUnipotent}. In fact, while Deligne--Lusztig theory allows us to control the representation theory of finite reductive groups, it is not sufficient to control the representation theory of $e$-chain stabilisers $\G^F_\sigma$. However, observe that the stabiliser $\G_\sigma^F$ contains the finite reductive group $\L(\sigma)^F$ as a normal subgroup. Therefore, we can first use Deligne--Lusztig theory to study the characters of $\L(\sigma)^F$ and then apply Clifford theory via $\G^F$-block isomorphisms of character triples to control the characters of $\G^F_\sigma$ (see Proposition \ref{prop:iEBC going up} and Proposition \ref{prop:Parametrisation for e-chains stabilisers} for further details).

In order to achieve the latter step, we need to make Deligne--Lusztig theory and, more precisely, $e$-Harish-Chandra theory for unipotent characters compatible with $\G^F$-block isomorphisms of character triples. This ideas was first suggested by the author in \cite[Parametrisation B]{Ros-Generalized_HC_theory_for_Dade} and further studied in \cite{Ros-Clifford_automorphisms_HC}. Our next result, which is a key ingredient in the proofs of Theorem \ref{thm:Main iUnipotent} and Theorem \ref{thm:Main Unipotent}, establishes this conjectured parametrisation in the unipotent case under the assumption specified above. This can also be seen as an extension of the parametrisation introduced by Brou\'e, Malle and Michel in \cite[Theorem 3.2 (2)]{Bro-Mal-Mic93} to the language of $\G^F$-block isomorphisms of character triples.

\begin{theoA}
\label{thm:Main Parametrisation for simple groups}
For every unipotent $e$-cuspidal pair $(\L,\lambda)$ of $(\G,F)$ there exists an $\aut_\mathbb{F}(\G^F)_{(\L,\lambda)}$-equivariant bijection
\[\Omega_{(\L,\lambda)}^\G:\E\left(\G^F,(\L,\lambda)\right)\to\irr\left(\n_\G(\L)^F\enspace\middle|\enspace\lambda\right)\]
that preserves the $\ell$-defect of characters and such that
\[\left(X_\chi,\G^F,\chi\right)\iso{\G^F}\left(\n_{X_\chi}(\L),\n_{\G}(\L)^F,\Omega_{(\L,\lambda)}^\G(\chi)\right)\]
for every $\chi\in\E(\G^F,(\L,\lambda))$ and where $X:=\G^F\rtimes \aut_\mathbb{F}(\G^F)$.
\end{theoA}

The proof of Theorem \ref{thm:Main Parametrisation for simple groups}, and therefore of Theorem \ref{thm:Main iUnipotent} and Theorem \ref{thm:Main Unipotent}, partially relies on certain conditions on the extendibility of characters of $e$-split Levi subgroups that were first introduced to settle the inductive conditions for the McKay Conjecture and the Alperin--McKay Conjecture, and then further studied in the context of Parametrisation B of \cite{Ros-Generalized_HC_theory_for_Dade} (see the exact statement given in \cite[Definition 5.2]{Ros-Clifford_automorphisms_HC}). These conditions were obtain, under certain assumptions, for groups of type $\bf{A}$, $\bf{B}$ and $\bf{C}$ in the papers \cite{Bro-Spa20}, \cite{Bro22} and \cite{Bro-Ruh23} respectively. Nonetheless, a version of these results is expected to hold in general and hence we believe that the above theorems, obtained here for types $\bf{A}$, $\bf{B}$ and $\bf{C}$ with respect to an odd prime $\ell$, will extend to the general case as well.

\subsection{Structure of the paper}

The paper is organised as follows. In Section \ref{sec:Notations} we introduce the necessary notation and recall the main definitions and results used throughout the paper. Furthermore, in Section \ref{sec:Pseudo-unipotent} we introduce the notion of pseudo-unipotent character (see Definition \ref{def:Pseudo unipotent}) and prove a result on the regularity of blocks covering those containing such characters. Next, in Section \ref{sec:Compatibility with isomorphisms of character triples} we start working towards a proof of Theorem \ref{thm:Main Parametrisation for simple groups}. First, in Section \ref{sec:equivarance and maximal ext} we consider certain equivariance properties that can be established in the presence of extendibility conditions for characters of $e$-split Levi subgroups. Here, we also present a candidate for the bijection $\Omega_{(\L,\lambda)}^\G$ required by Theorem \ref{thm:Main Parametrisation for simple groups}. Next, in Section \ref{sec:Construction of isomorphisms of character triples} we construct the required $\G^F$-block isomorphisms of character triples. Using these results, we can then prove Theorem \ref{thm:Main Parametrisation for simple groups} in Section \ref{sec:Proof of Main parametrisation}. The following step is to extend the parametrisation of unipotent $e$-Harish-Chandra series in the group $\G$, as given by Theorem \ref{thm:Main Parametrisation for simple groups}, to a parametrisation of pseudo-unipotent $e$-Harish-Chandra series in $F$-stable Levi subgroups $\K$ of $(\G,F)$. This is done in Theorem \ref{thm:Parametrisation for reductive groups}. Once this is established, in Section \ref{sec:above} we exploit the theory of $\G^F$-block isomorphisms to obtain bijections above $e$-Harish-Chandra series that are required to control the representation theory of the $e$-chain stabilisers $\G_\sigma^F$. A more detailed analysis of the characters of $\G_\sigma^F$ is carried out in Section \ref{sec:e-chains}. In particular, we obtain a parametrisation of the character sets $\uch^d(B_\sigma,(\M,\mu))$ in Proposition \ref{prop:Parametrisation for e-chains stabilisers}. Finally, in Section \ref{sec:Proof of iUnipotent} and Section \ref{sec:Counting} we apply these results to prove Theorem \ref{thm:Main iUnipotent} and Theorem \ref{thm:Main Unipotent} respectively.

\section{Notation and background material}
\label{sec:Notations}

\subsection{Characters and blocks of finite groups}

We recall some standard notation from representation theory of finite groups as can be found in \cite{Isa76} and \cite{Nav98}, for instance. Let $\irr(G)$ the set of ordinary irreducible characters. If $N\unlhd G$ and $\vartheta\in\irr(N)$, then we denote by $\irr(G\mid \vartheta)$ the set of irreducible characters of $G$ that lie above $\vartheta$. More generally, if $\mathcal{S}$ is a subset of irreducible characters of $N$, then we denote by $\irr(G\mid \mathcal{S})$ the union of the sets $\irr(G\mid \vartheta)$ for $\vartheta\in\mathcal{S}$, that is, the set of irreducible characters of $G$ that lie above some character in the set $\mathcal{S}$.

Next, we denote by $G_\vartheta$ the stabiliser of the irreducible character $\vartheta\in\irr(N)$ under the conjugacy action of $G$ and say that $\vartheta$ is $G$-invariant if $G=G_\vartheta$. In this case, we say that $(G,N,\vartheta)$ is a \textit{character triple}. These objects provide important information in the study of Clifford theory and play a crucial role in many aspects of the local-global conjectures. Of paramount importance is the introduction of certain binary relations on the set of character triples. We refer the reader to \cite[Chapter 5 and 10]{Nav18} and \cite{Spa18} for a more detailed introduction to these ideas and for the necessary background on projective representations. The binary relation considered here was introduced in \cite[Definition 3.6]{Spa17} and is known as \textit{$N$-block isomorphism of character triples}, denoted by $\iso{N}$. This equivalence relation has further been studied in \cite{Ros22}.

In order to construct $N$-block isomorphisms of character triples, it is often useful to prove certain results on the extendibility of characters. Here, we introduce the notion of \textit{maximal extendibility} (see \cite[Definition 3.5]{Mal-Spa16}) that will be considered in the following sections. Let $N\unlhd G$ be finite groups and consider $\mathcal{S}$ a subset of irreducible characters of $N$. Then, we say that \emph{maximal extendibility} holds for the set $\mathcal{S}$ with respect to the inclusion $N\unlhd G$ if every character $\vartheta\in\mathcal{S}$ extends to its stabiliser $G_\vartheta$. More precisely, we can specify an \emph{extension map}
\begin{equation}
\label{eq:Maximal extendibility}
\Lambda:\mathcal{S}\to\coprod\limits_{N\leq H\leq G}\irr(H)
\end{equation}
that sends each character $\vartheta\in\mathcal{S}$ to an extension $\Lambda(\vartheta)$ of $\vartheta$ to the stabiliser $G_\vartheta$.

Next, we consider modular representation theory with respect to a fixed prime number $\ell$. For $\chi\in\irr(G)$, there exist unique non-negative integers $d(\chi)$, called the \textit{$\ell$-defect} of $\chi$, such that $\ell^{d(\chi)}=|G|_\ell/\chi(1)_\ell$ and where for an integer $n$ we denote by $n_\ell$ the largest power of $\ell$ that divides $n$. For any $d\geq 0$, let $\irr^d(G)$ be the set of irreducible characters $\chi$ of $G$ that satisfy $d(\chi)=d$ and denote by $\k^d(G)$ its cardinality. Associated to the prime $\ell$, we also have the set of \textit{Brauer $\ell$-blocks} of $G$. Each block is uniquely determined by the central functions $\lambda_B$ (see \cite[p. 49]{Nav98}). For every $\chi\in\irr(G)$, we denote by $\bl(\chi)$ the unique block that satisfies $\chi\in\irr(\bl(\chi))$. Furthermore, if $H\leq G$ and $b$ is a block of $H$, then $b^G$ denotes the block of $G$ obtained via Brauer's induction (when it is defined). If $B$ is a block of $G$ and $d\geq 0$, then let $\irr^d(B)$ be the set of irreducible characters belonging to the block $B$ and having defect $d$. The cardinality of $\irr^d(B)$ is denoted by $\k^d(B)$.

We conclude this introductory section with an analogue of \cite[Problem 5.3]{Isa76} for blocks that will be used in the sequel.

\begin{lem}
\label{lem:Block induction and linear characters}
Let $H\leq G$ be finite groups and consider blocks $b$ of $H$ and $B$ of $G$. If $\zeta$ is a linear character of $G$, then:
\begin{enumerate}
\item there are blocks $b\cdot\zeta_H$ of $H$ and $B\cdot\zeta$ of $G$ satisfying
\[\irr(b\cdot\zeta_H)=\{\psi\zeta_H\mid \psi\in\irr(b)\}\hspace{15pt}\text{and}\hspace{15pt}\irr(B\cdot\zeta)=\{\chi\zeta\mid \chi\in\irr(B)\};\]
\item If $b^G=B$, then $(b\cdot\zeta_H)^G=B\cdot \zeta$.
\end{enumerate}
\end{lem}

\begin{proof}
The first point is \cite[Lemma 2.1]{Riz18}. Next, let $g\in G$ and denote by $\cl_G(g)$ the $G$-conjugacy class of $g$ and by $\cl_G(g)^+$ the corresponding conjugacy class sum in the group algebra. Since the intersection $\cl_G(g)\cap H$ is a union of $H$-conjugacy classes, we can find $h_1,\dots,h_n\in\cl_G(g)\cap H$ such that
\[\cl_G(g)\cap H=\coprod\limits_{i=1}^n\cl_H(h_i)\]
and where $n$ is zero if $\cl_G(g)\cap H$ is empty. In particular, observe that $\zeta(h_i)=\zeta(g)$ since $\lambda$ is a class function of $G$. Now, using the notation of \cite[p.87]{Nav98} we obtain
\begin{align*}
\lambda_{B\cdot\zeta}\left(\cl_G(g)^+\right)&=\lambda_B\left(\cl_G(g)^+\right)\overline{\zeta(g)}
\\
&=\lambda_b^G\left(\cl_G(g)^+\right)\overline{\zeta(g)}
\\
&=\sum\limits_{i=1}^n\lambda_b\left(\cl_H(h_i)^+\right)\overline{\zeta(g)}
\\
&=\sum\limits_{i=1}^n\lambda_b\left(\cl_H(h_i)^+\right)\overline{\zeta_H(h_i)}
\\
&=\sum\limits_{i=1}^n\lambda_{b\cdot\zeta_H}\left(\cl_H(h_i)^+\right)=\lambda_{b\cdot\zeta_H}^G\left(\cl_G(g)\right)
\end{align*}
where for every algebraic integer $\alpha$ of $\mathbb{C}$ we denote by $\overline{\alpha}$ its reduction modulo a maximal ideal containing the prime $\ell$ (see \cite[Chapter 2]{Nav98}). This shows that $B\cdot \zeta=(b\cdot\zeta_H)^G$ and we are done.
\end{proof}

\subsection{Finite reductive groups and unipotent characters}

Let $\G$ be a connected reductive group defined over an algebraic closure of a field of positive characteristic $p$ different from $\ell$ and consider a Frobenius endomorphism $F:\G\to \G$ associated with an $\mathbb{F}_q$-structure for a power $q$ of $p$. The set of $\mathbb{F}_q$-rational points on the variety $\G$ is denoted by $\G^F$ and is called a \textit{finite reductive group}. By abuse of notation we also refer to the pair $(\G,F)$ as a finite reductive group.

Let $\L$ be a Levi subgroup of a parabolic subgroup $\P$ of $\G$ and assume that $\L$ (but not necessarily $\P$) is $F$-stable. Using $\ell$-adic cohomology, Deligne--Lusztig \cite{Del-Lus76} and Lusztig \cite{Lus76} defined a $\mathbb{Z}$-linear map
\[\R_{\L\leq \P}^\G:\mathbb{Z}\irr\left(\L^F\right)\to\mathbb{Z}\irr\left(\G^F\right)\]
with adjoint
\[^*\R_{\L\leq \P}^\G:\mathbb{Z}\irr\left(\G^F\right)\to\mathbb{Z}\irr\left(\L^F\right).\]
The exact definition can be found in \cite[Section 8.3]{Cab-Eng04}. These maps are known to be independent of the choice of the parabolic subgroup $\P$ in almost all cases (see \cite{Bon-Mic10} and \cite{Tay18}) and, in particular, in those considered in this paper. Therefore, we will always omit $\P$ and denote $\R_{\L\leq \P}^\G$ simply by $\R_\L^\G$. Next, using Deligne--Lusztig induction we define the \textit{unipotent characters} of $\G^F$. These are the irreducible characters $\chi$ of $\G^F$ that appear as an irreducible constituent of the virtual character $\R_{\T}^{\G}(1_\T)$ for some $F$-stable maximal torus $\T$ of $\G$. The set of unipotent characters of $\G^F$ is denoted by $\uch(\G^F)$ and its cardinality by $\k_{\rm u}(\G^F)$. Similarly, if $B$ is a block of $\G^F$ and $d$ a non-negative integer, then $\k_{\rm u}^d(B)$ denotes the cardinality of the intersection $\uch(\G^F)\cap \irr^d(B)$.  

\subsection{$e$-Harish-Chandra theory for unipotent characters}

Denote by $e$ the multiplicative order of $q$ modulo $\ell$, if $\ell$ is odd, or modulo $4$, if $\ell=2$. In this section, we collect the main results of $e$-Harish-Chandra theory for unipotent characters. This was first introduced
by Fong and Srinivasan \cite{Fon-Sri86} for classical groups and then further developed by Brou\'e, Malle and Michel \cite{Bro-Mal-Mic93} for unipotent characters. The compatibility of this theory with Brauer $\ell$-blocks was described by Cabanes and Enguehard in \cite{Cab-Eng94} for good primes and completed by Enguehard \cite{Eng00} for bad primes. These results also provide a description of the characters belonging to unipotent blocks (see \cite[Theorem (iii)]{Cab-Eng94}). Another description of these characters was provided by the author in \cite{Ros-Generalized_HC_theory_for_Dade} under certain resctrictions on the prime $\ell$ (see also \cite[Remark 4.14]{Ros-Generalized_HC_theory_for_Dade} for a comparison between the two descriptions). We refer the reader to the monographs \cite{Cab-Eng04} and \cite{Gec-Mal20} for a more complete account of this beautiful theory.

The theory of $\Phi_e$-subgroups that constitutes the foundation of $e$-Harish-Chandra theory was introduced in \cite{Bro-Mal92}. Following their terminology, we say that an $F$-stable torus $\S$ of $\G$ is a \textit{$\Phi_e$-torus} if its order polynomial is a power of the $e$-th cyclotomic polynomial, that is, if $P_{(\S,F)}=\Phi_e^{n}$ for some integer $n$ and where $\Phi_e$ denotes the $e$-th cyclotomic polynomial (see \cite[Definition 13.3]{Cab-Eng04}). Then, we say that a Levi subgroup $\L$ of $\G$ is an \textit{$e$-split Levi subgroup} if there exists a $\Phi_e$-torus $\S$ such that $\L=\c_\G(\S)$. More precisely, we say that $\L$ is an $e$-split Levi subgroup of $(\G,F)$ to emphasise the role of the Frobenius endomorphism $F$. Observe that, for any torus $\T$, there exists a unique maximal $\Phi_e$-torus of $\T$ denoted by $\T_{\Phi_e}$ (see \cite[Proposition 13.5 3.4]{Cab-Eng04}). Then, it can be shown that an $F$-stable Levi subgroup $\L$ of $\G$ is $e$-split if and only if $\L=\c_\G(\z^\circ(\L)_{\Phi_e})$ (see, for instance, \cite[Proposition 3.5.5]{Gec-Mal20}).

Next, recall that $(\L,\lambda)$ is a \textit{unipotent $e$-cuspidal pair} of $(\G,F)$ if $\L$ is an $e$-split Levi subgroup of $(\G,F)$ and $\lambda\in\irr(\L^F)$ satisfies $^*\R_{\M}^\L(\lambda)=0$ for every $e$-split Levi subgroup $\M<\L$. A character $\lambda$ with the property above is said to be a \textit{unipotent $e$-cuspidal character} of $\L^F$. We denote by $\CP(\G,F)$ the set of unipotent $e$-cuspidal pairs of $(\G,F)$ and by $\k_{\rm c,u}(\G^F)$ the number of unipotent $e$-cuspidal characters of $\G^F$. Moreover, we define the \textit{$e$-Harish-Chandra series} associate to the $e$-cuspidal pair $(\L,\lambda)$ to be the set of irreducible constituents of the virtual character $\R_\L^\G(\lambda)$, denoted by $\E(\G^F,(\L,\lambda))$.

Unipotent characters where parametrised by Brou\'e, Malle and Michel \cite[Theorem 3.2]{Bro-Mal-Mic93} by using $e$-Harish-Chandra theory. Their description can be divided into two parts. First, each unipotent character lies in a unique $e$-Harish-Chandra series, that is,
\[\uch\left(\G^F\right)=\coprod\limits_{(\L,\lambda)}\E\left(\G^F,(\L,\lambda)\right)\]
where $(\L,\lambda)$ runs over a set of representatives for the action of $\G^F$ on the set of unipotent $e$-cuspidal pairs of $(\G,F)$ as explained in \cite[Theorem 3.2 (1)]{Bro-Mal-Mic93}. This is a well known fact and will be used throughout the paper without further reference. As a consequence of the partition above, it now remains to parametrise the unipotent $e$-Harish-Chandra series. If $(\L,\lambda)$ is a unipotent $e$-cuspidal pair, we denote by $W_\G(\L,\lambda)^F:=\n_\G(\L)^F_\lambda/\L^F$ the corresponding \textit{relative Weyl group}. Then, \cite[Theorem 3.2 (2)]{Bro-Mal-Mic93} parametrises the characters in an $e$-Harish-Chandra series in terms of the characters in the relative Weyl group by showing the existence of a bijection
\begin{equation}
\label{eq:e-HC series and relative Weyl group}
\irr\left(W_\G(\L,\lambda)^F\right)\to \E(\G^F,(\L,\lambda)).
\end{equation}
In Section \ref{sec:Compatibility with isomorphisms of character triples} we reformulate \eqref{eq:e-HC series and relative Weyl group} in order to obtain Theorem \ref{thm:Main Parametrisation for simple groups}.

Unipotent $e$-Harish-Chandra series are also used to parametrise the so-called \textit{unipotent blocks}, that is, those blocks that contain unipotent characters. This is the main result of \cite{Cab-Eng94}. More precisely, if $\ell$ is odd and good for $\G$, with $\ell\neq 3$ if ${^3{\bf{D}}_4}$ is an irreducible rational component of $(\G,F)$, then for every $\ell$-block $B$ of $\G^F$ there exists a unipotent $e$-cuspidal pair $(\L,\lambda)$, with $(\L,\lambda)$ unique up to $\G^F$-conjugation, such that all the irreducible constituents of $\R_\L^\G(\lambda)$ belongs to the block $B$. In this case, we write $B=b_{\G^F}(\L,\lambda)$ and we also have
\[\uch(\G^F)\cap \irr\left(b_{\G^F}(\L,\lambda)\right)=\E(\G^F,(\L,\lambda)).\]
Moreover, \cite[Proposition 3.3 (ii) and Proposition 4.2]{Cab-Eng94} imply that $\bl(\lambda)^{\G^F}=B$.

\subsection{Pseudo-unipotent characters}
\label{sec:Pseudo-unipotent}

We denote by $(\G^*,F^*)$ a group in duality with $(\G,F)$ with respect to a choice of an $F$-stable maximal torus $\T$ of $\G$ and an $F^*$-stable maximal torus $\T^*$ of $\G^*$. If $\tau:\G_{\rm sc}\to [\G,\G]$ is a simply connected covering (see \cite[Remark 1.5.13]{Gec-Mal20}), then there exists an isomorphisms between the abelian groups
\begin{align*}
\z\left(\G^*\right)^{F^*}&\to\irr\left(\G^F/\tau(\G_{\rm sc})\right)
\\
z &\mapsto \hat{z}_{\G}
\end{align*}
according to \cite[(8.19)]{Cab-Eng04}. Notice that, if $\L$ is an $F$-stable Levi subgroup of $\G$, then its dual $\L^*$ is an $F^*$-stable Levi subgroup of $\G^*$ and we have $\z(\G^*)^{F^*}\leq \z(\L^*)^{F^*}$. In particular, every element $z\in\z(\G^*)^{F^*}$ defines a linear characters of $\hat{z}_{\L}$ and restriction of characters yields the equality
\[(\hat{z}_\G)_{\L^F}=\hat{z}_\L.\]
In the next definition, we consider characters that are obtained by multiplying these linear characters with unipotent characters.

\begin{defin}
\label{def:Pseudo unipotent}
Let $(\K,F)$ be a finite reductive group and consider a Levi subgroup of $\L\leq\K$ and an irreducible character $\theta\in\irr(\L^F)$. We say that $\theta$ is \textit{$(\K,F)$-pseudo-unipotent} if there exists an element $z\in\z(\K^*)^{F^*}$ such that $\theta\hat{z}_{\L}$ is unipotent. Moreover, for every unipotent character $\lambda\in\uch(\L^F)$, we denote by $\ps_\K(\lambda)$ the set of $(\K,F)$-pseudo-unipotent characters of $\L^F$ of the form $\lambda\hat{z}_{\L}$ for some $z\in\z(\K^*)^{F^*}$. Moreover, we denote by $\ps_{\K}(\L^F)$ the set of all $(\K,F)$-pseudo unipotent characters of $\L^F$. When the group $\K$ coincides with $\L$, we denote the set of characters $\ps_\L(\L^F)$ simply by $\ps(\L^F)$.
\end{defin}

In accordance with the terminology introduced above, we say that an $e$-Harish-Chandra series of $(\K,F)$ is pseudo-unipotent if it is of the form $\E(\K^F,(\L,\nu))$ for some $\nu\in\ps_\K(\lambda)$ and where $(\L,\lambda)$ is a unipotent $e$-cuspidal pair of $(\K,F)$. In this case, we also say that $(\L,\nu)$ is a \textit{pseudo-unipotent $e$-cuspidal pair}. We define the union of all the series associated to characters in $\ps_\K(\lambda)$ by $\E(\K^F,(\L,\ps_\K(\lambda)))$. Since
\[\R_\L^\K(\lambda\hat{z}_\L)=\R_\L^\K(\lambda)\hat{z}_\K\]
for every $z\in\z(\K^*)^{F^*}$ by \cite[(8.20)]{Cab-Eng04}, we deduce that the elements of the pseudo-unipotent $e$-Harish-Chandra series $\E(\K^F,(\L,\lambda\hat{z}))$ are exactly the irreducible characters of the form $\varphi\hat{z}_{\K}$ for some unipotent character $\varphi\in\E(\K^F,(\L,\lambda))$. Moreover, we point out that $\lambda$ is the unique unipotent character in the set $\ps_\K(\lambda)$ according to \cite[Proposition 8.26]{Cab-Eng04}. Similarly, the unipotent characters in the set $\E(\K^F,(\L,\ps_\K(\lambda)))$ are those in the series $\E(\K^F,(\L,\lambda))$.

Our next lemma, shows that blocks covering pseudo-unipotent characters are regular as defined in \cite[p.210]{Nav98}.

\begin{lem}
\label{lem:Block induction and regularity}
Let $\L$ be an $F$-stable Levi subgroup of $\G$ and suppose that $\ell$ is odd and good for $\G$. For every $\L^F\leq H\leq \n_\G(\L)^F$ and every character $\vartheta\in\irr(H)$ lying above some pseudo-unipotent character in $\ps(\L^F)$, the block $\bl(\vartheta)$ is $\L^F$-regular. In particular, the Brauer induced block $\bl(\vartheta)^H$ is defined and is the unique block of $H$ covering $\bl(\vartheta)$.
\end{lem}

\begin{proof}
Let $\varphi\in\uch(\L^F)$ and $z\in\z(\L^*)^{F^*}$ such that $\varphi\hat{z}_{\L}$ lies below the character $\vartheta$ and chose a unipotent $e$-cuspidal pair $(\M,\mu)$ of $\L$ such that $\varphi\in\E(\L^F,(\M,\mu))$. In particular, $\bl(\varphi)=b_{\L^F}(\M,\mu)$ according to \cite{Cab-Eng94}. If $Q:=\z(\M)_\ell^F$, then $\M^F=\c_{\G^F}(Q)$ according to \cite[Proposition 3.3 (ii)]{Cab-Eng94}. Moreover, observe that \cite[Proposition 4.2]{Cab-Eng94} implies that $\bl(\varphi)=b_{\L^F}(\M,\mu)=\bl(\mu)^{\L^F}$ while \cite[Lemma 2.1]{Riz18} implies that $\bl(\varphi)$ and $\bl(\varphi\hat{z}_\L)$ have the same defect groups. Now, applying \cite[Lemma 4.13 and Theorem 9.26]{Nav98}, we can find defect groups $D_\vartheta$, $D_\varphi$ and $D_\mu$ of $\bl(\vartheta)$, $\bl(\varphi)$ and $\bl(\mu)$ respectively with the property that $D_\mu\leq D_\varphi\leq D_\vartheta$. Since $Q\leq\O_\ell(\M^F)\leq D_\mu$ by \cite[Theorem 4.8]{Nav98}, we deduce that $Q\leq D_\vartheta$ and hence $\c_H(D_\vartheta)\leq \c_H(Q)=\M^F\leq \L^F$. By \cite[Lemma 9.20]{Nav98} we conclude that the block $\bl(\vartheta)$ is $\L^F$-regular. The second part of the lemma now follows from \cite[Theorem 9.19]{Nav98}.
\end{proof}

\section{Compatibility with isomorphisms of character triples}
\label{sec:Compatibility with isomorphisms of character triples}

The aim of this section is to show how the bijection \eqref{eq:e-HC series and relative Weyl group} can be made compatible with isomorphisms of character triples and with the action of automorphisms. This property was first suggested by the author in \cite[Parametrisation B]{Ros-Generalized_HC_theory_for_Dade} and further studied in \cite{Ros-Clifford_automorphisms_HC}. Our Theorem \ref{thm:Main Parametrisation for simple groups} gives a solution of this conjectured result for unipotent $e$-Harish-Chandra series and groups of type $\bf{A}$, $\bf{B}$ and $\bf{C}$. Before proceeding further, we show how the parametrisation \eqref{eq:e-HC series and relative Weyl group} can be reformulated in a more convenient form. For this, let $(\L,\lambda)$ be a unipotent $e$-cuspidal pair of $(\G,F)$ and assume that $\wh{\lambda}$ is an extension of $\lambda$ to the stabiliser $\n_\G(\L)^F_\lambda$. Then, by applying Gallagher's theorem \cite[Corollary 6.17]{Isa76} and the Clifford correspondence \cite[Theorem 6.11]{Isa76} we obtain a bijection
\begin{align*}
\irr\left(W_\G(\L,\lambda)^F\right)&\to\irr\left(\n_\G(\L)^F\enspace\middle|\enspace\lambda\right)
\\
\eta&\mapsto \left(\wh{\lambda}\eta\right)^{\n_\G(\L)^F}
\end{align*}
and therefore \eqref{eq:e-HC series and relative Weyl group} holds if an only if there exists a bijection
\begin{equation}
\label{eq:e-HC series, parametrization with normalizer}
\E(\G^F,(\L,\lambda))\to\irr\left(\n_\G(\L)^F\enspace\middle|\enspace\lambda\right).
\end{equation}
This new reformulation will allow us to introduce the aforementioned compatibility with isomorphisms of character triple isomorphisms.

\subsection{Equivariance and maximal extendibility}
\label{sec:equivarance and maximal ext}

In this section, we consider some equivariance properties for the parametrisation \eqref{eq:e-HC series, parametrization with normalizer} which are related to maximal extendibility (see \eqref{eq:Maximal extendibility}) of unipotent characters.

As in the previous sections, consider a connected reductive group $\G$ with a Frobenius endomorphism $F:\G\to \G$ defining an $\mathbb{F}_q$-structure on $\G$. We denote by $\aut_\mathbb{F}(\G^F)$ the set of those automorphisms of $\G^F$ obtained by restricting some bijective morphism of algebraic groups $\sigma:\G\to\G$ that commutes with $F$ to the set of $\mathbb{F}_q$-rational points $\G^F$. Notice that the restriction of such a morphism $\sigma$ to $\G^F$, which by abuse of notation we denote again by $\sigma$, is an automorphism of the finite group $\G^F$. We refer the reader to \cite[Section 2.4]{Cab-Spa13} for further details. In particular, observe that any morphism $\sigma$ with the properties above is determined by its restriction to $\G^F$ up to a power of $F$ and hence it follows that $\aut_\mathbb{F}(\G^F)$ acts on the set of $F$-stable closed connected subgroups of $\G$. Then, given an $F$-stable closed connected subgroup $\H$ of $\G$, we can define the set $\aut_\mathbb{F}(\G^F)_\H$ consisting of those automorphisms $\sigma$ as above that stabilise the algebraic group $\H$.

Now, let $\ell$ be a prime number not dividing $q$ and denote by $e$ the order of $q$ modulo $\ell$ or $q$ modulo $4$ if $\ell=2$. In order to control the action of automorphism on unipotent $e$-Harish-Chandra series, we exploit a result of Cabanes and Sp\"ath. More precisely, in \cite[Theorem 3.4]{Cab-Spa13} it was shown that the parametrisation given by Brou\'e, Malle and Michel in \cite[Theorem 3.2 (2)]{Bro-Mal-Mic93} commutes with the action of those automorphisms in the set $\aut_\mathbb{F}(\G^F)$. Notice that the statement of \cite[Theorem 3.4]{Cab-Spa13} only considers unipotent $e$-cuspidal pairs $(\L,\lambda)$ where $\L$ is a minimal $e$-split Levi subgroups (which is enough for the purpose of dealing with the McKay Conjecture). However, their proof works for the general case as well.

\begin{prop}
\label{prop:Equivariant BMM}
For every unipotent $e$-cuspidal pair $(\L,\lambda)$ of $(\G,F)$ there exists an $\aut_\mathbb{F}(\G^F)_{(\L,\lambda)}$-equivariant bijection
\[I_{(\L,\lambda)}^\G:\irr\left(W_\G(\L,\lambda)^F\right)\to\E\left(\G^F,(\L,\lambda)\right)\]
such that
\[I^\G_{(\L,\lambda)}(\eta)(1)_\ell=\left|\hspace{1pt}\G^F:\n_\G(\L,\lambda)^F\hspace{1pt}\right|_\ell\cdot\lambda(1)_\ell\cdot\eta(1)_\ell\]
for every $\eta\in\irr(W_\G(\L,\lambda)^F)$.
\end{prop}

\begin{proof}
This follows from the proof of \cite[Theorem 3.4]{Cab-Spa13}. See also \cite[Theorem 3.4]{Ros-Clifford_automorphisms_HC}.
\end{proof}

As explained at the beginning of this section, if the character $\lambda$ extends to the stabiliser $\n_\G(\L)^F_\lambda$, then we can use the bijection \eqref{eq:e-HC series and relative Weyl group} to obtain \eqref{eq:e-HC series, parametrization with normalizer}. A similar argument can be used to include the equivariance property described above and obtain an equivariant version of \eqref{eq:e-HC series, parametrization with normalizer}. Observe that, by the discussion on automorphisms above, it follows that the group $\aut_\mathbb{F}(\G^F)$ acts on the set of $e$-cuspidal pairs $(\L,\lambda)$ and therefore we can define the stabiliser $\aut_\mathbb{F}(\G^F)_{(\L,\lambda)}$. Furthermore, recall that we denote by $d(\chi)$ the $\ell$-defect of an irreducible character $\chi$.

\begin{cor}
\label{cor:Equivariant BMM, extended}
Let $(\L,\lambda)$ be a unipotent $e$-cuspidal pair of $(\G,F)$ and suppose that $\lambda$ has an extension $\lambda^\diamond\in\irr(\n_\G(\L)_\lambda^F)$ which is additionally $\aut_{\mathbb{F}}(\G^F)_{(\L,\lambda)}$-invariant. Then there exists an $\aut_{\mathbb{F}}(\G^F)_{(\L,\lambda)}$-equivariant bijection
\[\Omega_{(\L,\lambda)}^\G:\E\left(\G^F,(\L,\lambda)\right)\to\irr\left(\n_\G(\L)^F\enspace\middle|\enspace\lambda\right)\]
such that
\[d(\chi)=d\left(\Omega^\G_{(\L,\lambda)}(\chi)\right)\]
for every $\chi\in\E(\G^F,(\L,\lambda))$.
\end{cor}

\begin{proof}
Consider the bijection $I^{\G}_{(\L,\lambda)}$ given by Proposition \ref{prop:Equivariant BMM} and define the map
\begin{align*}
\Omega_{(\L,\lambda)}^\G:\E\left(\G^F,(\L,\lambda)\right)&\to\irr\left(\n_\G(\L)^F\enspace\middle|\enspace \lambda\right)
\\
I_{(\L,\lambda)}^\G(\eta)&\mapsto\left(\lambda^\diamond\eta\right)^{\n_\G(\L)^F}
\end{align*}
for every $\eta\in\irr(W_\G(\L,\lambda)^F)$ and where $\lambda^\diamond$ is the extension of $\lambda$ to $\n_\G(\L)_\lambda^F$ given in the statement. This is a well defined bijection by the Clifford correspondence \cite[Theorem 6.11]{Isa76} and Gallagher's theorem \cite[Corollary 6.17]{Isa76}. Moreover, for every $\alpha\in\aut_\mathbb{F}(\G^F)$ such that $(\L,\lambda)^\alpha=(\L,\lambda)$ and every $\eta\in\irr(W_\G(\L,\lambda)^F)$ we have
\begin{align*}
\left(\left(\lambda^\diamond\eta\right)^{\n_\G(\L)^F}\right)^\alpha&=\left(\left(\lambda^\diamond\eta\right)^\alpha\right)^{\n_\G(\L)^F}
\\
&=\left(\lambda^\diamond\eta^\alpha\right)^{\n_\G(\L)^F}
\end{align*}
because $\alpha$ stabilises $\lambda^\diamond$. On the other hand
\[I_{(\L,\lambda)}^\G(\eta)^\alpha=I_{(\L,\lambda)}^\G\left(\eta^\alpha\right)\]
by the properties of $I_{(\L,\lambda)}^\G$ and hence we conclude that $\Omega_{(\L,\lambda)}^\G$ is $\aut_{\mathbb{F}}(\G^F)_{(\L,\lambda)}$-equivariant. Furthermore, if we consider $\eta\in\irr(W_\G(\L,\lambda)^F)$ and define the characters $\chi:=I_{(\L,\lambda)}^\G(\eta)$ and $\psi:=(\lambda^\diamond\eta)^{\n_\G(\L)^F}$, then the degree formula from Proposition \ref{prop:Equivariant BMM} implies that
\begin{equation*}
\label{eq:Equivariant BMM, extended, 1}
\ell^{d(\chi)}=\dfrac{\left|\hspace{1pt}\G^F\hspace{1pt}\right|_\ell}{\chi(1)_\ell}=\dfrac{\left|\hspace{1pt}\n_\G(\L,\lambda)^F\hspace{1pt}\right|_\ell}{\lambda(1)_\ell\cdot\eta(1)_\ell}=\dfrac{\left|\hspace{1pt}\n_\G(\L)^F\hspace{1pt}\right|_\ell}{\psi(1)_\ell}=\ell^{d(\psi)}
\end{equation*}
and hence we deduce that $d(\chi)=d(\psi)$ as required.
\end{proof}

Next, we consider a \textit{regular embedding} $\G\leq \wt{\G}$ as defined in \cite[(15.1)]{Cab-Eng04}. Then, $\wt{\G}$ is a connected reductive group with connected centre and whose derived subgroup coincides with that of $\G$, that is, $[\wt{\G},\wt{\G}]=[\G,\G]$. In particular, observe that $\wt{\G}=\z(\wt{\G})\G$, that $\G$ is normal in  $\wt{\G}$ and that the quotient $\wt{\G}/\G$ is an abelian group. Moreover, for every Levi subgroup $\L$ of $\G$, we deduce that $\wt{\L}:=\z(\wt{\G})\L$ is a Levi subgroup of $\wt{\G}$ and that $\L\leq \wt{\L}$ is again a regular embedding. These observations will be used throughout this paper without further reference.

We also recall that, according to \cite[Proposition 13.20]{Dig-Mic91}, restriction of characters yields a bijection between the unipotent characters of $\wt{\G}^F$ and those of $\G^F$. In particular, every unipotent character of $\G^F$ is $\wt{\G}^F$-invariant. Using this observation, we can compare the relative Weyl groups in $\wt{\G}^F$ with those in $\G^F$.

\begin{lem}
\label{lem:Relative weyl groups and regular embeddings}
Let $(\L,\lambda)$be a unipotent $e$-cuspidal pair of $(\G,F)$, set $\wt{\L}=\L\z(\wt{\G})$ and consider a unipotent extension $\wt{\lambda}$ of $\lambda$ to $\wt{\L}^F$. Then, $\n_{\wt{\G}}(\L)^F_\lambda=\n_{\wt{\G}}(\L)^F_{\wt{\lambda}}$ and we have $W_{\wt{\G}}(\wt{\L},\wt{\lambda})^F\simeq W_\G(\L,\lambda)^F$.
\end{lem}

\begin{proof}
Since $\wt{\lambda}$ extends $\lambda$, it is clear that the stabiliser $\n_{\wt{\G}}(\L)^F_{\wt{\lambda}}$ is contained in $\n_{\wt{\G}}(\L)^F_\lambda$. On the other hand, let $x\in \n_{\wt{\G}}(\L)^F_\lambda$ and observe that $\wt{\lambda}^x$ is a unipotent character of $\wt{\L}^F$ that restricts to $\lambda^x=\lambda$. Then, \cite[Proposition 13.20]{Dig-Mic91} implies that $\wt{\lambda}^x=\wt{\lambda}$ and therefore $x\in\n_{\wt{\G}}(\L)^F_{\wt{\lambda}}$. From this, we also conclude that $\n_{\wt{\G}}(\L)^F_{\wt{\lambda}}=\wt{\L}^F\n_{\G}(\L)^F_\lambda$ and therefore that $W_{\wt{\G}}(\wt{\L},\wt{\lambda})^F\simeq W_\G(\L,\lambda)^F$.
\end{proof}

As a consequence of the lemma above, we show that when $\lambda$ extends to its stabiliser $\n_\G(\L)^F_\lambda$, then every irreducible character of $\n_\G(\L)$ that lies above $\lambda$ is $\n_{\wt{\G}}(\L)^F$-invariant and extends to $\n_{\wt{\G}}(\L)^F$.

\begin{cor}
\label{cor:Invariant characters over unipotent}
Let $(\L,\lambda)$ be a unipotent $e$-cuspidal pair of $(\G,F)$ and suppose that $\lambda$ has an extension $\lambda^\diamond\in\irr(\n_\G(\L)_\lambda^F)$. Then every character of $\n_\G(\L)^F$ lying above $\lambda$ extends to $\n_{\wt{\G}}(\L)^F$.
\end{cor}

\begin{proof}
To start, we fix a unipotent extension $\wt{\lambda}$ of $\lambda$ to $\wt{\L}^F$ and recall that $\n_{\wt{\G}}(\L)^F_\lambda=\n_{\wt{\G}}(\L)^F_{\wt{\lambda}}$ according to Lemma \ref{lem:Relative weyl groups and regular embeddings}. Then, applying \cite[Lemma 4.1 (a)]{Spa10} we deduce that there exists an extension $\wt{\lambda}^\diamond$ of $\lambda^\diamond$ to $\n_{\wt{\G}}(\L)^F_\lambda$ that also extends $\wt{\lambda}$. Consider now an irreducible character $\psi$ of $\n_\G(\L)^F$ lying above $\lambda$. By Gallagher's theorem \cite[Corollary 6.17]{Isa76} and the Clifford correspondence \cite[Theorem 6.11]{Isa76}, it follows that there exists an irreducible character $\eta$ of the relative Weyl group $W_\G(\L,\lambda)^F$ such that $\psi$ is induced from the irreducible character $\psi_0:=\eta\lambda^\diamond$. Moreover, by using Lemma \ref{lem:Relative weyl groups and regular embeddings}, we have $W_{\wt{\G}}(\wt{\L},\wt{\lambda})^F\simeq W_\G(\L,\lambda)^F$. Then, $\eta$, viewed as a character of $\n_\G(\L)^F_\lambda$, admits an extension, say $\wt{\eta}$, to $\n_{\wt{\G}}(\L)^F_\lambda$. Now, define $\wt{\psi}_0:=\wt{\eta}\wt{\lambda}^\diamond$ and observe that $\wt{\psi}_0$ lies above $\wt{\lambda}$. By the Clifford correspondence, it follows that the character $\wt{\psi}$ of $\n_{\wt{\G}}(\L)^F$ induced from $\wt{\psi}_0$ is irreducible and therefore, applying \cite[Problem 5.2]{Isa76}, we conclude that $\wt{\psi}$ extends $\psi$. The proof is now complete.
%Consider a character $\psi\in\irr(\n_\G(\L)^F\mid \lambda)$ and fix an element $g\in\n_{\wt{\G}}(\L)^F$. Notice that $\n_{\wt{\G}}(\L)^F=\wt{\L}^F\n_\G(\L)^F$ and write $g=xy$ for some element $x\in\n_\G(\L)^F$ and $y\in\wt{\L}^F$. Since $\psi$ is $\n_\G(\L)^F$ equivariant, we obtain $\psi^g=\psi^y$ and it remains to show that $\psi^y=\psi$. First, observe that $y$ stabilises $\lambda$ according to \cite[Proposition 13.20]{Dig-Mic91} and hence $y\in\aut_\mathbb{F}(\G^F)_{(\L,\lambda)}$. Then, consider the map $\Omega_{(\L,\lambda)}$ given by Corollary \ref{cor:Equivariant BMM, extended} and let $\chi\in\E(\G^F,(\L,\lambda))$ such that $\Omega^\G_{(\L,\lambda)}(\chi)=\psi$. Since $\Omega_{(\L,\lambda)}^\G$ is $\aut_\mathbb{F}(\G^F)_{(\L,\lambda)}$-equivariant, it is enough to show that $\chi^y=\chi$. But the latter property is satisfied by \cite[Proposition 13.20]{Dig-Mic91} because $\chi$ is unipotent.
\end{proof}

We can now construct a parametrisation of unipotent $e$-Harish-Chandra series in the group $\wt{\G}^F$ which agrees with the bijection $\Omega_{(\L,\lambda)}^\G$ from Corollary \ref{cor:Equivariant BMM, extended} via restriction of characters.

\begin{prop}
\label{prop:Bijection commuting diagram}
Let $(\L,\lambda)$ be a unipotent $e$-cuspidal pair of $(\G,F)$ and suppose that $\lambda$ has an extension $\lambda^\diamond\in\irr(\n_\G(\L)_\lambda^F)$ which is additionally $\aut_{\mathbb{F}}(\G^F)_{(\L,\lambda)}$-invariant. If $\wt{\lambda}$ is a unipotent extension of $\lambda$ to $\wt{\L}^F$, then there exists a bijection $\wt{\Omega}_{(\wt{\L},\wt{\lambda})}^{\wt{\G}}$ making the following diagram commute
\begin{center}
\begin{tikzcd}
\E\left(\wt{\G}^F,(\wt{\L},\wt{\lambda})\right)\arrow[r, "\wt{\Omega}_{(\wt{\L},\wt{\lambda})}^{\wt{\G}}"]\arrow[d, swap, "{\rm Res}^{\wt{\G}^F}_{\G^F}"] &\irr\left(\n_{\wt{\G}}(\L)^F\enspace\middle|\enspace \wt{\lambda}\right)\arrow[d, "{\rm Res}^{\n_{\wt{\G}}(\L)^F}_{\n_\G(\L)^F}"]
\\
\E\left(\G^F,(\L,\lambda)\right)\arrow[r, swap, "\Omega_{(\L,\lambda)}^\G"] &\irr\left(\n_{\G}(\L)^F\enspace\middle|\enspace \lambda\right)
\end{tikzcd}
\end{center}
and where $\Omega_{(\L,\lambda)}^\G$ is the bijection given by Corollary \ref{cor:Equivariant BMM, extended}.
\end{prop}

\begin{proof}
First, observe that $\lambda$ has an extension $\wt{\lambda}$ to $\wt{\L}^F$ according to \cite[Proposition 13.20]{Dig-Mic91}. Moreover, restrictions from $\wt{\G}^F$ to $\G^F$ induces a bijection from the set $\E(\wt{\G}^F,(\wt{\L},\wt{\lambda}))$ to $\E(\G^F,(\L,\lambda))$ according to \cite[Proposition 3.1]{Cab-Eng94}. Next, consider a character $\psi\in\irr(\n_\G(\L)^F)$ lying above $\lambda$ and observe that $\psi$ admits an extension $\wt{\psi}_0\in\irr(\n_{\wt{\G}}(\L)^F)$ by Corollary \ref{cor:Invariant characters over unipotent}. Let $\wt{\lambda}_0$ be an irreducible constituent of the restriction $\wt{\psi}_{0,\wt{\L}^F}$ and notice that $\wt{\lambda_0}$ is an extension of $\lambda$ since $\wt{\L}^F/\L^F$ is abelian. Now, Gallagher's theorem \cite[Corollary 6.17]{Isa76} implies that there exists a linear character $\nu\in\irr(\wt{\L}^F/\L^F)$ such that $\wt{\lambda}_0\nu=\wt{\lambda}$. Since $\n_{\wt{\G}}(\L)^F/\n_\G(\L)^F\simeq \wt{\L}^F/\L^F$ we can identify $\nu$ with its extension to $\n_{\wt{\G}}(\L)^F$. Then, it follows that the character $\wt{\psi}:=\wt{\psi}_0\nu$ is an extension of $\psi$ to $\n_{\wt{\G}}(\L)^F$ lying above $\wt{\lambda}$. Then the assignment $\psi\mapsto\wt{\psi}$ defines a bijection between $\irr(\n_\G(\L)^F\mid \lambda)$ and $\irr(\n_{\wt{\G}}(\L)^F\mid \wt{\lambda})$ whose inverse is given by restriction of characters. We can now define
\[\wt{\Omega}_{(\wt{\L},\wt{\lambda})}^{\wt{\G}}\left(\wt{\chi}
\right):=\wt{\psi}\]
for every $\wt{\chi}\in\E(\wt{\G}^F,(\wt{\L},\wt{\lambda}))$ and $\wt{\psi}\in\irr(\n_{\wt{\G}}(\L)^F\mid \wt{\lambda})$ whenever $\Omega_{(\L,\lambda)}^\G(\wt{\chi}_{\G^F})=\wt{\psi}_{\n_\G(\L)^F}$.
\end{proof}

\subsection{Construction of $\G^F$-block isomorphisms of character triples}
\label{sec:Construction of isomorphisms of character triples}

From now on, we assume that $\G$ is simple, simply connected and of type $\bf{A}$, $\bf{B}$ or $\bf{C}$. Furthermore, we suppose that $\ell$ is odd and denote by $e$ the order of $q$ modulo $\ell$.

We now give a more explicit construction of the group of automorphism $\aut_\mathbb{F}(\G^F)$. Fix a maximally split torus $\T_0$ contained in an $F$-stable Borel subgroup $\B_0$ of $\G$. This choice corresponds to a set of graph automorphisms $\gamma:\G\to \G$ and a field endomorphism $F_0:\G\to \G$. More precisely, if we consider the set of simple roots $\Delta\subseteq \Phi(\G,\T_0)$ corresponding to the choice $\T_0\subseteq \B_0$, then we have an automorphism $\gamma:\G\to \G$ given by $\gamma(x_\alpha(t)):=x_{\gamma(\alpha)}(t)$ for every $t\in\mathbb{G}_{\rm a}$ and $\alpha\in\pm\Delta$ and where $\gamma$ is a symmetry of the Dynkin diagram of $\Delta$, while $F_0(x_\alpha(t)):=x_{\alpha}(t^p)$ for every $t\in\mathbb{G}_{\rm a}$ and $\alpha\in\Phi(\G,\T_0)$. Here, we denote by $x_\alpha:\mathbb{G}_{\rm a}\to \G$ a one-parameter subgroup corresponding to $\alpha\in\Phi(\G,\T_0)$. We define the subgroup $\mathcal{A}$ of $\aut_\mathbb{F}(\G^F)$ generated by the graph and field automorphisms described above.

In addition, we choose our regular embedding $\G\leq \wt{\G}$ to be defined in such a way that the graph and field automorphisms extends to $\wt{\G}$ (see, for instance, \cite[Section 2B]{Mal-Spa16}). In particular, the group $\mathcal{A}$ acts via automorphisms on $\wt{\G}^F$ and we can form the external semidirect product $\wt{\G}^F\rtimes \mathcal{A}$ which acts on $\G^F$. It turns out that $\wt{\G}^F\rtimes \mathcal{A}$ and $\aut_\mathbb{F}(\G^F)$ induce the same set of automorphisms on the finite group $\G^F$ (see, for instance, \cite[Section 2.5]{GLS}).

Throughout this section, we consider a fixed unipotent $e$-cuspidal pair $(\L,\lambda)$ of $(\G,F)$ and a unipotent extension $\wt{\lambda}$ of $\lambda$ to $\wt{\L}^F$ (whose existence is ensured by \cite[Proposition 13.20]{Dig-Mic91}) where, as always, we define $\wt{\L}:=\L\z(\wt{\G})$. In the next lemma, we show that the hypothesis of Corollary \ref{cor:Equivariant BMM, extended} is satisfied under our assumptions.

\begin{lem}
\label{lem:Extension to the graph and field locally}
There exists an extension $\lambda^\diamond$ of $\lambda$ to $\n_\G(\L)_\lambda^F$ that is $(\wt{\G}^F\mathcal{A})_{(\L,\lambda)}$-invariant.
\end{lem}

\begin{proof}
Using \cite[Theorem 4.3 (i)]{Bro-Spa20}, \cite[Theorem 1.2 (a)]{Bro22} and the results of \cite{Bro-Ruh23}, we obtain an extension $\lambda^\diamond$ of $\lambda$ to $\n_\G(\L)_\lambda^F$ which is $(\G^F\mathcal{A})_{(\L,\lambda)}$-invariant. Since $(\wt{\G}^F\mathcal{A})_{(\L,\lambda)}=\wt{\L}(\G^F\mathcal{A})_{(\L,\lambda)}$ it suffices to show that $\lambda^\diamond$ is $\wt{\L}^F$-invariant. However, the latter assertion follows immediately from the fact that $\lambda^\diamond$ extends to $\n_{\wt{\G}}(\L)^F_{\lambda}$ according to Lemma \ref{lem:Relative weyl groups and regular embeddings} and \cite[Lemma 4.1 (a)]{Spa10}.
\end{proof}

As an immediate consequence of the lemma above, we deduce that every character of $\n_\G(\L)^F$ lying above $\lambda$ extends to $\n_{\wt{\G}}(\L)^F$. This can be considered as a local analogue of \cite[Proposition 13.20]{Dig-Mic91}.

\begin{lem}
\label{lem:Extension of local unipotent characters to diagonal}
Every irreducible character of $\n_\G(\L)^F$ lying above $\lambda$ extends to $\n_{\wt{\G}}(\L)^F$. 
\end{lem}

\begin{proof}
This follows from Corollary \ref{cor:Invariant characters over unipotent} whose hypothesis is satisfied by Lemma \ref{lem:Extension to the graph and field locally}.
\end{proof}

We point out that, under our assumptions, every irreducible character of $\n_\G(\L)^F$ lying above $\lambda$ extends to its stabiliser in $\n_{\wt{\G}}(\L)^F$ because the quotient $\n_{\wt{\G}}(\L)^F/\n_\G(\L)^F$ is cyclic according to \cite[Proposition 1.7.5]{Gec-Mal20}. However, in the lemma above we are also showing, using independent methods, that each such character is $\n_{\wt{\G}}(\L)^F$-invariant.

Using Lemma \ref{lem:Extension to the graph and field locally}, we can now define bijections $\Omega:=\Omega_{(\L,\lambda)}^\G$ and $\wt{\Omega}_{(\wt{\L},\wt{\lambda})}^{\wt{\G}}$ as described in Corollary \ref{cor:Equivariant BMM, extended} and Proposition \ref{prop:Bijection commuting diagram} respectively. In what follows, we consider the sets of characters $\mathcal{G}:=\E(\G^F,(\L,\lambda))$, $\mathcal{L}:=\irr(\n_\G(\L)^F\mid \lambda)$, $\wt{\mathcal{G}}:=\E(\wt{\G}^F,(\wt{\L},\wt{\lambda}))$ and $\wt{\mathcal{L}}:=\irr(\n_{\wt{\G}}(\L)^F\mid \wt{\lambda})$. Our next aim is to show that the parametrisation $\Omega$ is compatible with $\G^F$-block isomorphisms of character triples. We start by checking the group theoretic properties required for the existence of such isomorphisms (see \cite[Remark 3.7 (i)]{Spa17}).

\begin{lem}
\label{lem:Frattini with inertial subgroups}
For every $\chi\in\mathcal{G}$ and $\psi:=\Omega(\chi)\in\mathcal{L}$ we have $(\wt{\G}^F\mathcal{A})_{\L,\chi}=(\wt{\G}^F\mathcal{A})_{\L,\psi}$ and $\wt{\G}^F\mathcal{A}_\chi=\G^F(\wt{\G}^F\mathcal{A})_{\L,\psi}$.
\end{lem}

\begin{proof}
We argue as in the proof of \cite[Lemma 4.2]{Ros-Generalized_HC_theory_for_Dade}. To start, we observe that $(\wt{\G}^F\mathcal{A})_{(\L,\lambda),\chi}=(\wt{\G}^F\mathcal{A})_{(\L,\lambda),\psi}$ since the map $\Omega$ is $(\wt{\G}^F\mathcal{A})_{(\L,\lambda)}$-equivariant. Set $U(\chi):=(\wt{\G}^F\mathcal{A})_{\L,\chi}$ and $U(\psi):=(\wt{\G}^F\mathcal{A})_{\L,\psi}$. First, consider $x\in U(\chi)$ and observe that according to \cite[Theorem 3.2 (1)]{Bro-Mal-Mic93} there exists $y\in\n_\G(\L)^F$ such that $(\L,\lambda)^{xy}=(\L,\lambda)$. In particular, $xy\in(\wt{\G}^F\mathcal{A})_{(\L,\lambda),\chi}=(\wt{\G}^F\mathcal{A})_{(\L,\lambda),\psi}$ and hence $x\in U(\psi)$ since $\psi^y=\psi$. This shows that $U(\chi)\leq U(\psi)$. On the other hand, suppose that $x\in U(\psi)$. By Clifford's theorem there exists $y\in\n_\G(\L)^F$ such that $\lambda^{xy}=\lambda$ and so $xy\in(\wt{\G}^F\mathcal{A})_{\L,\psi}=(\wt{\G}^F\mathcal{A})_{\L,\chi}$. Since $\chi^y=\chi$ we deduce that $x\in U(\chi)$ and hence $U(\chi)=U(\psi)$. To conclude, it is enough to show that $\wt{\G}^F\mathcal{A}_\chi=\G^F U(\chi)$. First, notice that $\G^F U(\chi)\leq \wt{\G}^F\mathcal{A}_\chi$ since $\chi$ is $\wt{\G}^F$-invariant. On the other hand, for $x\in\wt{\G}^F\mathcal{A}_\chi$ we know that $(\L,\lambda)^{x}$ is $\G^F$-conjugate to $(\L,\lambda)$ thanks to \cite[Theorem 3.2 (1)]{Bro-Mal-Mic93}. Therefore, we obtain $x\in \G^F U(\chi)$ and as explained above this concludes the proof.
\end{proof}

We now apply Lemma \ref{lem:Frattini with inertial subgroups} to show that the map $\wt{\Omega}$ satisfies some useful equivariance properties. Before doing so, we need to introduce some notation. For this purpose, consider a pair $(\G^*,F^*)$ dual to $(\G,F)$ and a pair $(\wt{\G}^*,F^*)$ dual to $(\wt{\G},F)$. Let $i^*:\wt{\G}^*\to \G^*$ be the surjection induced by duality from the inclusion $\G\leq \wt{\G}$ and observe that $\ker(i^*)=\z(\wt{\G}^*)$ since $\G$ is simply connected (see \cite[Section 15.1]{Cab-Eng04}). As shown in \cite[(15.2)]{Cab-Eng04}, there exists an isomorphism
\begin{align}
\ker(i^*)^F&\to \irr\left(\wt{\G}^F/\G^F\right)\label{central kernel and linear characters}
\\
z&\mapsto \wh{z}_{\wt{\G}}\nonumber
\end{align}
Furthermore, if $\L$ is an $F$-stable Levi subgroup of $\G$ and $z\in\ker(i^*)$, then we define $\wh{z}_{\wt{\L}}$ to be the restriction of $\wh{z}_{\wt{\G}}$ to $\wt{\L}^F$ and $\wh{z}_{\n_{\wt{\G}}(\L)}$ to be the restriction of $\wh{z}_{\wt{\G}}$ to $\n_{\wt{\G}}(\L)^F$. We set $\mathcal{K}:=\ker(i^*)$ and obtain an action of the group $\mathcal{K}$ on the characters of $\wt{\G}^F$, $\wt{\L}^F$ and $\n_{\wt{\G}}(\L)^F$ as defined in \cite[Definition 2.1]{Ros-Clifford_automorphisms_HC}. Moreover, we consider the external semidirect product $(\wt{\G}^F\mathcal{A})\ltimes \mathcal{K}$ given by defining $z^x$ as the unique element of $\mathcal{K}$ corresponding to the character $(\wh{z}_{\wt{\G}})^x$ of the quotient $\wt{\G}^F/\G^F$ via the isomorphism specified in \eqref{central kernel and linear characters}, whenever $x\in \wt{\G}^F\mathcal{A}$ and $z\in \mathcal{K}$. Then, for every $F$-stable Levi subgroup $\L$ of $\G$, we obtain an action of $(\wt{\G}^F\mathcal{A})_\L\ltimes \mathcal{K}$ on the irreducible characters of $\wt{\L}^F$ and $\n_{\wt{\G}}(\L)^F$. We denote by $((\wt{\G}^F\mathcal{A})_\L\ltimes \mathcal{K})_{\wt{\lambda}}$ the stabiliser of $\wt{\lambda}\in\irr(\wt{\L}^F)$. In particular, it follows that $((\wt{\G}^F\mathcal{A})_\L\ltimes \mathcal{K})_{\wt{\lambda}}$ acts on the sets of characters $\wt{\mathcal{G}}$ and $\wt{\mathcal{L}}$. Next, we show that the bijection $\wt{\Omega}$ is compatible with this action.

\begin{lem}
\label{lem:Equivariant over Omega}
The bijection $\wt{\Omega}$ is $(\n_{\wt{\G}}(\L)^F(\wt{\G}^F\mathcal{A})_{(\L,\lambda)}\rtimes \mathcal{K})_{\wt{\lambda}}$-equivariant.
\end{lem}

\begin{proof}
Let $\wt{\chi}\in\wt{\mathcal{G}}$ and $\wt{\psi}\in\wt{\mathcal{L}}$. By the definition of $\wt{\Omega}$, we have $\wt{\Omega}(\wt{\chi})=\wt{\psi}$ if and only if $\Omega(\chi)=\psi$ where $\chi:=\wt{\chi}_{\G^F}$ and $\psi:=\wt{\psi}_{\n_\G(\L)^F}$. Now, if we consider $g\in\n_{\wt{\G}}(\L)^F$, $x\in (\wt{\G}^F\mathcal{A})_{(\L,\lambda)}$ and $z\in\mathcal{K}$ such that $(gx,z)$ stabilises ${\wt{\lambda}}$, then we obtain
\[\wt{\Omega}\left(\wt{\chi}^{(gx,z)}\right)=\wt{\psi}^{(gx,z)}\]
if and only if
\begin{equation}
\label{eq:Equivariant over Omega, 1}
\Omega\left(\left(\wt{\chi}^{(gx,z)}\right)_{\G^F}\right)=\left(\wt{\psi}^{(gx,z)}\right)_{\n_\G(\L)^F}.
\end{equation}
However, since the restriction of $\wt{\chi}^{(gx,z)}$ to $\G^F$ coincides with $\chi^x$ and the restriction of $\wt{\psi}^{(gx,z)}$ to $\n_\G(\L)^F$ coincides with $\psi^x$, we deduce that the equality in \eqref{eq:Equivariant over Omega, 1} holds by the equivariant properties of $\Omega$ as described in Corollary \ref{cor:Equivariant BMM, extended}.
\end{proof}

One of the main ingredients for the construction of the projective representations needed to obtain $\G^F$-block isomorphisms of character triples is given by the following two lemmas on maximal extendibility.

\begin{lem}
\label{lem:Global extendibility in types A and C}
Maximal extendibility holds for $\mathcal{G}$ with respect to the inclusion $\G^F\unlhd \G^F\mathcal{A}$, that is, every character $\chi\in\mathcal{G}$ extends to $\G^F\mathcal{A}_\chi$.
\end{lem}

\begin{proof}
If $\G$ is of type $\bf{B}$ or $\bf{C}$ then the result follows from \cite[Corollary 11.22]{Isa76} since $\mathcal{A}$ is cyclic. Then, we can assume that $\G$ is of type $\bf{A}$ in which case the result follows from \cite[Theorem 4.1]{Cab-Spa17I} (see also \cite[Theorem 2.4]{Mal08}).
\end{proof}

The local version of the lemma above is a consequence of the results obtained in \cite{Bro-Spa20}.

\begin{lem}
\label{lem:Local extendibility in types A and C}
Maximal extendibility holds for $\mathcal{L}$ with respect to the inclusion $\n_\G(\L)^F\unlhd (\G^F\mathcal{A})_\L$, that is, every character $\psi\in\mathcal{L}$ extends to $(\G^F\mathcal{A})_{\L,\psi}$.
\end{lem}

\begin{proof}
As in the proof of Lemma \ref{lem:Global extendibility in types A and C}, it is enough to prove the result in the case where $\G$ is of type $\bf{A}$. In fact, if $\G$ is of type $\bf{B}$ or $\bf{C}$, then the quotient $(\G^F\mathcal{A})_{(\L,\psi)}/\n_\G(\L)^F$ is cyclic because it it a subquotient of $\mathcal{A}$. Now, if $\G$ is of type $\bf{A}$ the result follows from \cite[Theorem 1.2]{Bro-Spa20}.
\end{proof}

Finally, we can start constructing isomorphisms of character triples for the bijection $\Omega$. As a first step, we obtain a weaker isomorphism, know as \textit{$\G^F$-central isomorphism of character triples} and denoted by $\isoc{\G^F}$, whose requirements are given by \cite[Remark 3.7 (i)-(iii)]{Spa17} and replacing the condition on defect groups by imposing that $\c_G(N)\leq H_1\cap H_2$ with the notations used there. We refer the reader to \cite[Definition 3.3.4]{Ros-Thesis} for a precise definition.

\begin{prop}
\label{prop:Parametrisation, central isomorphisms}
For every $\chi\in\mathcal{G}$ and $\psi:=\Omega(\chi)\in\mathcal{psi}$ we have
\[\left(\wt{\G}^F\mathcal{A}_\chi,\G^F,\chi\right)\isoc{\G^F}\left((\wt{\G}^F\mathcal{A})_{\L,\psi},\n_\G(\L)^F,\psi\right).\]
\end{prop}

\begin{proof}
We start by constructing projective representations associated with $\chi$ and $\psi$. According to Proposition \ref{prop:Bijection commuting diagram} we can find a unipotent extension $\wt{\chi}\in\wt{\mathcal{G}}$ of $\chi$ to $\wt{\G}^F$. Furthermore, by Lemma \ref{lem:Global extendibility in types A and C} there exists an extension $\chi'$ of $\chi$ to $\G^F\mathcal{A}_\chi$. Let $\wt{\mathcal{D}}_{\rm glo}$ be a representation of $\wt{\G}^F$ affording $\wt{\chi}$ and $\mathcal{D}'_{\rm glo}$ a representation of $\G^F\mathcal{A}_\chi$ affording $\chi'$. Now, \cite[Lemma 2.11]{Spa12} implies that
\[\Pr_{\rm glo}:\left(\wt{\G}^F\mathcal{A}\right)_\chi\to \GL_{\chi(1)}(\mathbb{C})\]
defined by $\Pr_{\rm glo}(x_1x_2):=\wt{\mathcal{D}}_{\rm glo}(x_1)\mathcal{D}'_{\rm glo}(x_2)$ for every $x_1\in\wt{\G}^F$ and $x_2\in \G^F\mathcal{A}_\chi$ is a projective representation associated with $\chi$. Next, observe that $\wt{\psi}:=\wt{\Omega}(\wt{\chi})\in\wt{\mathcal{L}}$ is an extension of $\psi$ to $\n_{\wt{\G}}(\L)^F$ and consider an extension $\psi'$ of $\psi$ to $(\G^F\mathcal{A})_{\L,\psi}$ given by Lemma \ref{lem:Local extendibility in types A and C}. Let $\wt{\mathcal{D}}_{\rm loc}$ be a representation of $\n_{\wt{\G}}(\L)^F$ affording $\wt{\psi}$ and $\mathcal{D}'_{\rm loc}$ a representation of $(\G^F\mathcal{A})_{\L,\psi}$ affording $\psi'$. Once again, \cite[Lemma 2.11]{Spa12} shows that the map
\[\Pr_{\rm loc}:\left(\wt{\G}^F\mathcal{A}\right)_{\L,\psi}\to \GL_{\psi(1)}(\mathbb{C})\]
given by $\Pr_{\rm loc}(x_1x_2):=\wt{\mathcal{D}}_{\rm loc}(x_1)\mathcal{D}'_{\rm loc}(x_2)$ for every $x_1\in\n_{\wt{\G}}(\L)^F$ and $x_2\in (\G^F\mathcal{A})_{\L,\psi}$ is a projective representation associated with $\psi$. We denote by $\alpha_{\rm glo}$ and $\alpha_{\rm loc}$ the factor set of $\Pr_{\rm glo}$ and $\Pr_{\rm loc}$ respectively. As explained in the proof of \cite[Theorem 4.3]{Ros-Clifford_automorphisms_HC}, in order to prove that $\alpha_{\rm glo}$ coincides with $\alpha_{\rm loc}$ via the isomorphism $\wt{\G}^F\mathcal{A}_\chi/\G^F\simeq (\wt{\G}^F\mathcal{A})_{\L,\psi}/\n_\G(\L)^F$, it suffices to show that
\begin{equation}
\label{eq:Parametrisation, central isomorphisms, 1}
(\mu_x^{\rm glo})_{\n_{\wt{\G}}(\L)^F}=\mu_x^{\rm loc}
\end{equation}
for every $x\in (\G^F\mathcal{A})_{\L,\chi}$ and where $\mu_x^{\rm glo}\in\irr(\wt{\G}^F/\G^F)$ and $\mu_x^{\rm loc}\in\irr(\n_{\wt{\G}}(\L)^F/\n_\G(\L)^F)$ are determined by Gallagher's theorem (see \cite[Corollary 6.17]{Isa76}) via the equalities $\wt{\chi}=\mu_x^{\rm glo}\wt{\chi}^x$ and $\wt{\psi}=\mu_x^{\rm loc}\wt{\psi}^x$ respectively. Because $(\G^F\mathcal{A})_{\L,\chi}=\n_\G(\L)^F(\G^F\mathcal{A})_{(\L,\lambda),\chi}$, we may assume that $x$ stabilises $\lambda$. Let $z\in\mathcal{K}$ such that $\mu_x^{\rm glo}=\hat{z}_{\wt{\G}}$ and observe that $(x,z)$ is an element of $(\G^F\mathcal{A})_{(\L,\lambda),\chi}\rtimes \mathcal{K}$ that stabilises $\wt{\chi}$. Then, applying \cite[Theorem 3.2 (1)]{Bro-Mal-Mic93}, we deduce that $\wt{\lambda}$ and $\wt{\lambda}^{(x,z)}$ are $\n_{\wt{\G}}(\L)^F$-conjugate and we may choose $g\in\n_{\wt{\G}}(\L)^F$ such that $\wt{\lambda}=(\wt{\lambda}^{(x,z)})^g=\wt{\lambda}^{(xg,z)}$. In other words
\[(xg,z)\in\left(\n_{\wt{\G}}(\L)^F(\wt{\G}^F\mathcal{A})_{(\L,\lambda)}\rtimes \mathcal{K}\right)_{\wt{\lambda}}\]
and thus Lemma \ref{lem:Equivariant over Omega} implies that the equality $\wt{\chi}=\wt{\chi}^{(xg,z)}$ holds if and only if $\wt{\psi}=\wt{\psi}^{(xg,z)}$. From this, we immediately deduce the equality required in \eqref{eq:Parametrisation, central isomorphisms, 1}.

Next, denote by $\zeta_{\rm glo}$ and $\zeta_{\rm loc }$ the scalar functions associated to $\Pr_{\rm glo}$ and $\Pr_{\rm loc}$ respectively. To conclude the proof, it remains to show that $\zeta_{\rm glo}$ and $\zeta_{\rm loc}$ coincide on $\c_{(\wt{\G}^F\mathcal{A})_{\chi}}(\G^F)=\z(\wt{\G}^F)$. As in the proof of \cite[Theorem 4.3]{Ros-Clifford_automorphisms_HC}, it is enough to show that the restrictions of $\wt{\chi}$ and $\wt{\psi}$ to $\z(\wt{\G}^F)$ are multiples of a common irreducible constituent. This follows from the fact that unipotent characters contain the center in their kernel. In fact, on one hand, $1_{\z(\wt{\G}^F)}$ is the unique irreducible constituent of $\wt{\chi}_{\z(\wt{\G}^F)}$ because $\wt{\chi}$ is unipotent. On the other hand, $\wt{\psi}$ lies above $\wt{\lambda}$ and, since $\z(\wt{\G}^F)\leq \z(\wt{\L}^F)$ and $\wt{\lambda}$ is unipotent, we deduce that $1_{\z(\wt{\G}^F)}$ is the unique irreducible constituent of $\wt{\psi}_{\z(\wt{\G}^F)}$. This completes the proof.
%By the character formula in \cite[Proposition 12.2 (i)]{Dig-Mic91} we deduce that $\R_{\wt{\L}}^{\wt{\G}}(\wt{\lambda})_{\z(\wt{\G}^F)}$ is a multiple of $\wt{\lambda}_{\z(\wt{\G}^F)}$ and thus
%\begin{equation}
%\label{eq:Parametrisation, central isomorphisms, 2}
%\irr\left(\wt{\chi}_{\z\left(\wt{\G}^F\right)}\right)=\irr\left(\wt{\lambda}_{\z\left(\wt{\G}^F\right)}\right)
%\end{equation}
%because $\wt{\chi}$ is an irreducible constituent of $\R_{\wt{\L}}^{\wt{\G}}(\wt{\lambda})$. On the other hand, since $\z(\wt{\G})\leq \wt{\L}$ and $\wt{\psi}$ lies above $\wt{\lambda}$, it follows that
%\begin{equation}
%\label{eq:Parametrisation, central isomorphisms, 3}
%\irr\left(\wt{\psi}_{\z\left(\wt{\G}^F\right)}\right)=\irr\left(\wt{\lambda}_{\z\left(\wt{\G}^F\right)}\right)
%\end{equation}
%Combining \eqref{eq:Parametrisation, central isomorphisms, 2} and \eqref{eq:Parametrisation, central isomorphisms, 3} we conclude that $\wt{\chi}$ and $\wt{\psi}$ lie above the same central character and hence that $\zeta_{\rm glo}=\zeta_{\rm loc}$. The proof is now complete.
\end{proof}

We conclude this section by verifying the remaining condition \cite[Remark 3.7 (iv)]{Spa17} and obtain the required $\G^F$-block isomorphisms of character triples for the map $\Omega$.

\begin{prop}
\label{prop:Parametrisation, block isomorphism}
For every $\chi\in\mathcal{G}$ and $\psi:=\Omega(\chi)\in\mathcal{psi}$ we have
\[\left(\wt{\G}^F\mathcal{A}_\chi,\G^F,\chi\right)\iso{\G^F}\left((\wt{\G}^F\mathcal{A})_{\L,\psi},\n_\G(\L)^F,\psi\right).\]
\end{prop}

\begin{proof}
By Proposition \ref{prop:Parametrisation, central isomorphisms} it is enough to check the block theoretic requirement given by \cite[Remark 3.7 (ii) and (iv)]{Spa17}. First, observe that under our assumption \cite[Proposition 3.3 (ii)]{Cab-Eng94} shows that $\L^F=\c_{\G^F}(E)$ where $E:=\z(\L)^F_\ell$. In particular, $\n_J(\L)=\n_J(E)$ for every $\G^F\leq J\leq \wt{\G}^F$. Furthermore, for every block $C_0$ of $\n_J(\L)$ and every defect group $D$ of $C_0$ we have $E\leq \O_\ell(\n_J(\L))\leq D$ and hence $\c_{\wt{\G}^F}(D)\leq \n_{\wt{\G}}(\L)^F$. Now, \cite[Theorem B]{Kos-Spa15} implies that for every block $C$ of $\n_{\wt{\G}}(\L)^F$ covering $C_0$, the induced blocks $B:=C^{\wt{\G}^F}$ and $B_0:=C_0^J$ are well-defined and $B$ covers $B_0$.

Let $\wt{\chi}\in\wt{\mathcal{G}}$ be an extension of $\chi$ and set $\wt{\psi}:=\wt{\Omega}(\wt{\chi})$. By Lemma \ref{lem:Block induction and regularity} the block of $\wt{C}$ of $\wt{\psi}$ coincides with the induced block $\bl(\wt{\lambda})^{\n_{\wt{\G}}(\L)^F}$. Furthermore, by \cite[Proposition 4.2]{Cab-Eng94} we know that the block $\wt{B}$ of $\wt{\chi}$ coincides with $b_{\wt{\G}^F}(\wt{\L},\wt{\lambda})=\bl(\wt{\lambda})^{\wt{\G}^F}$. Then, by the transitivity of block induction we get $\wt{B}=\wt{C}^{\wt{\G}^F}$. Consider now $\G^F\leq J\leq \wt{\G}^F$ as in the previous paragraph and notice that $\bl(\wt{\chi}_J)$ is the unique block of $J$ covered by $\wt{B}$. Now, since $\bl(\wt{\psi}_{\n_J(\L)})$ is covered by $\wt{C}$, we deduce that $\bl(\wt{\psi}_{\n_J(\L)})^J$ is covered by $\wt{B}$ and therefore
\begin{equation}
\label{eq:Parametrisation, block isomorphism, 1}
\bl\left(\wt{\chi}_J\right)=\bl\left(\wt{\psi}_{\n_J(\L)}\right)^J.
\end{equation}
As explained in the proof of \cite[Theorem 4.8]{Ros-Clifford_automorphisms_HC} we can now use \eqref{eq:Parametrisation, block isomorphism, 1} together with Proposition \ref{prop:Parametrisation, central isomorphisms} to conclude the proof via an application of \cite[Theorem 4.1 (i)]{Spa17}.
\end{proof}

\subsection{Proof of Theorem \ref{thm:Main Parametrisation for simple groups}}
\label{sec:Proof of Main parametrisation}

\begin{proof}[Proof of Theorem \ref{thm:Main Parametrisation for simple groups}]
The hypothesis of Corollary \ref{cor:Equivariant BMM, extended} is satisfied under our restrictions on $\G$ according to Lemma \ref{lem:Extension to the graph and field locally} and therefore we obtain an $\aut_\mathbb{F}(\G^F)_{(\L,\lambda)}$-equivariant bijection
\[\Omega^\G_{(\L,\lambda)}:\E\left(\G^F,(\L,\lambda)\right)\to\irr\left(\n_\G(\L)^F\enspace\middle|\enspace\lambda\right)\]
that preserves the $\ell$-defect of characters. Next, observe that the groups $\wt{\G}^F\mathcal{A}$ and $X:=\G^F\rtimes \aut_\mathbb{F}(\G^F)$ induce the same automorphisms on $\G^F$ according to the description given in \cite[Section 2.5]{GLS}. Then, by applying \cite[Theorem 5.3]{Spa17} and Proposition \ref{prop:Parametrisation, block isomorphism}, we conclude that
\[\left(X_\chi,\G^F,\chi\right)\iso{\G^F}\left(\n_{X}(\L)_\psi,\n_{\G}(\L)^F,\psi\right)\]
for every $\chi\in\E(\G^F,(\L,\lambda))$ and where $\psi:=\Omega_{(\L,\lambda)}^\G(\chi)$ and the proof is now complete.
\end{proof}

\section{Consequences of Theorem \ref{thm:Main Parametrisation for simple groups}}

In this section, we collect some consequences of Theorem \ref{thm:Main Parametrisation for simple groups}. First, we extend the parametrisation obtained in Theorem \ref{thm:Main Parametrisation for simple groups} from unipotent $e$-Harish-Chandra series of the simple group $\G$ to pseudo-unipotent (see Definition \ref{def:Pseudo unipotent}) $e$-Harish-Chandra series  of the Levi subgroups of $\G$. More precisely, for every $F$-stable Levi subgroup $\K$ of $\G$, we construct a parametrisation of the $e$-Harish-Chandra series associated to $e$-cuspidal pairs of the form $(\L,\lambda)$ for some $(\K,F)$-pseudo-unipotent character $\lambda\in\ps_\K(\L^F)$. In a second step, we construct character bijections above this parametrisation by exploiting results on isomorphisms of character triples (see Corollary \ref{cor:iEBC going up}). This will allow us to control the characters of $e$-chain stabilisers lying above pseudo-unipotent characters (see Proposition \ref{prop:Parametrisation for e-chains stabilisers}).

\subsection{Parametrisation of pseudo-unipotent characters of Levi subgroups}

Let $\K$ be an $F$-stable Levi subgroup of $\G$ and set $\K_0:=[\K,\K]$. Observe that since the group $\G$ is simply connected, the subgroup $\K_0$ is also simply connected according to \cite[Proposition 12.14]{Mal-Tes}. In addition, under our assumption on the type of $\G$, we deduce that the simple components of $\K_0$ can only be of some of the types $\bf{A}$, $\bf{B}$ or $\bf{C}$.

\begin{prop}
\label{prop:From rational components to K_0}
For every unipotent $e$-cuspidal pair $(\L_0,\lambda_0)$ of $(\K_0,F)$ there exists a defect preserving $\aut_\mathbb{F}(\K_0^F)_{(\L_0,\lambda_0)}$-equivariant bijection
\[\Omega^{\K_0}_{(\L_0,\lambda_0)}:\E\left(\K_0^F,(\L_0,\lambda_0)\right)\to\irr\left(\n_{\K_0}(\L_0)^F\hspace{1pt}\middle|\hspace{1pt} \lambda_0\right)\]
such that
\[\left(Y_{\vartheta},\K_0^F,\vartheta\right)\iso{\K_0^F}\left(\n_{Y_{\vartheta}}(\L_0),\n_{\K_0}(\L_0)^F,\Omega^{\K_0}_{(\L_0,\lambda_0)}(\vartheta)\right)
\]
for every $\vartheta\in\E(\K_0^F,(\L_0,\lambda_0))$ and where $Y:=\K_0^F\rtimes \aut_\mathbb{F}(\K_0^F)$.
\end{prop}

\begin{proof}
Notice that $\K_0$ is the direct product of simple algebraic groups $\K_1,\dots, \K_n$ and that the action of $F$ permutes the simple components $\K_i$. Denote the direct product of the simple components in each $F$-orbit by $\H_j$ for $j=1,\dots, t$. The $(\H_j,F)$ are the irreducible rational components of $(\K,F)$ and we have $\K_0^F=\H_1^F\times \cdots\times \H_t^F$. Similarly, if we define the intersections $\M_j:=\L_0\cap \H_j$, then we have a decomposition $\L_0^F=\M_1^F\times \cdots \times\M_t^F$. In particular, we can write $\lambda_0=\mu_1\times \cdots \times\mu_t$ with $\mu_j\in\irr(\M_j^F)$. In this case, notice that $(\M_j,\mu_j)$ is a unipotent $e$-cuspidal pair of $(\H_j,F)$. Next, suppose that $\H_j=\H_{j,1}\times\dots\times \H_{j,m_j}$ and observe that $\H_j^F\simeq \H_{j,1}^{F^{m_j}}$. By the discussion at the beginning of this section we know that $\H_{j,1}$ is a simple, simply connected group of type $\bf{A}$, $\bf{B}$ or $\bf{C}$ and hence it satisfies the assumptions of Theorem \ref{thm:Main Parametrisation for simple groups}. Then, via the isomorphism $\H_j^F\simeq \H_{j,1}^{F^{m_j}}$, we obtain an $\aut_{\mathbb{F}}(\H_j^F)_{(\M_j,\mu_j)}$-equivarint bijection
\[\Omega^{\H_j}_{(\M_j,\mu_j)}:\E\left(\H_j^F,(\M_j,\mu_j)\right)\to\irr\left(\n_{\H_j}(\M_j)^F\hspace{1pt}\middle|\hspace{1pt}\mu_j\right)\]
that preserves the defect of characters and such that
\begin{equation}
\label{eq:From rational components to K_0}
\left(Y_{j,\vartheta},\H_j^F,\vartheta\right)\iso{\H_j^F}\left(\n_{Y_{j,\vartheta}}(\M_j),\n_{\H_j}(\M_j)^F,\Omega^{\H_j}_{(\M_j,\mu_j)}(\vartheta)\right)
\end{equation}
for every $\vartheta\in\E(\H_j^F,(\M_j,\mu_j))$ and where $Y_j:=\H_j^F\rtimes \aut_\mathbb{F}(\H_j^F)$. Since the characters in the sets $\E(\K_0^F,(\L_0,\lambda_0))$ and $\irr(\n_{\K_0}(\L_0)^F\mid \lambda_0)$ are direct products of characters belonging to the sets $\E(\H_j^F,(\M_j,\mu_j))$ and $\irr(\n_{\H_j}(\M_j)^F\mid \mu_j)$ respectively, we obtain a bijection
\[\Omega_{(\L_0,\lambda_0)}^{\K_0}:\E\left(\K_0^F,(\L_0,\lambda_0)\right)\to\irr\left(\n_{\K_0}(\L_0)^F\hspace{1pt}\middle|\hspace{1pt} \lambda_0\right)\]
by setting
\[\Omega_{(\L_0,\lambda_0)}^{\K_0}\left(\vartheta_1\times\dots\times \vartheta_t\right):=\Omega_{(\M_1,\mu_1)}^{\H_1}(\vartheta_1)\times \dots\times \Omega_{(\M_t,\mu_t)}^{\H_t}(\vartheta_t)\]
for every $\vartheta_j\in\E(\H_j^F,(\M_j,\mu_j))$. Finally, arguing as in the proof of \cite[Proposition 6.5]{Ros-Generalized_HC_theory_for_Dade}, we deduce that the bijection $\Omega^{\K_0}_{(\L_0,\lambda_0)}$ preserves the defect of characters, is $\aut_\mathbb{F}(\K_0^F)_{(\L_0,\lambda_0)}$-equivariant, and, using \eqref{eq:From rational components to K_0}, it induces the $\K_0^F$-block isomorphisms of character triples required in the statement.
\end{proof}

In our next result, we replace the automorphism group $Y:=\K_0^F\rtimes \aut_\mathbb{F}(\K_0^F)$ with the group of automorphisms of $\G^F$ stabilising $\K$, that is, $X:=(\G^F\rtimes \aut_\mathbb{F}(\G^F))_\K$. To do so, we apply the so-called \textit{Butterfly theorem} \cite[Theorem 5.3]{Spa17} which basically states that, for any finite group $G$, the notion of $G$-block isomorphism of character triples only depends on the automorphisms induced on $G$.

\begin{cor}
\label{cor:From rational components to K_0}
Let $(\L_0,\lambda_0)$ be a unipotent $e$-cuspidal pair of $(\K_0,F)$. The map $\Omega_{(\L_0,\lambda_0)}^{\K_0}$ given by Proposition \ref{prop:From rational components to K_0} is $\aut_\mathbb{F}(\G^F)_{\K,(\L_0,\lambda_0)}$-equivariant and satisfies
\begin{equation}
\label{eq:cor From rational components to K_0, 1}
\left(X_{\vartheta},\K_0^F,\vartheta\right)\iso{\K_0^F}\left(\n_{X_{\vartheta}}(\L_0),\n_{\K_0}(\L_0)^F,\Omega^{\K_0}_{(\L_0,\lambda_0)}(\vartheta)\right)
\end{equation}
for every $\vartheta\in\E(\K_0^F,(\L_0,\lambda_0))$ and where $X:=(\G^F\rtimes \aut_\mathbb{F}(\G^F))_\K$.
\end{cor}

\begin{proof}
First, observe that $\aut_\mathbb{F}(\G^F)_\K$ is contained in $\aut_\mathbb{F}(\K_0^F)$ because $\K_0$ is an $F$-stable characteristic subgroup of $\K$. In particular, we deduce that the map $\Omega_{(\L_0,\lambda_0)}^{\K_0}$ is $\aut_\mathbb{F}(\G^F)_{\K,(\L_0,\lambda_0)}$-equivariant. Next, to obtain \eqref{eq:cor From rational components to K_0, 1}, we apply \cite[Lemma 3.8 and Theorem 5.3]{Spa17} to the isomorphism of character triples given by Proposition \ref{prop:From rational components to K_0} as explained in the proof of \cite[Corollary 6.8]{Ros-Generalized_HC_theory_for_Dade}.
\end{proof}

Isomorphisms of character triples play a fundamental role in representation theory of finite groups and in the study of the local-global conjectures. One of the most important consequences of the existence of isomorphisms of character triples is the possibility to lift character bijections. For instance, the main result of \cite{Nav-Spa14I}, shows how to apply this technique to construct bijections above characters of height zero in the context of the Alperin--McKay Conjecture \cite[Theorem B]{Nav-Spa14I}. The main consequence of this result, which follows from an argument introduced by Murai \cite{Mur12}, is a reduction theorem for the celebrated Brauer's Height Zero Conjecture \cite[Theorem A]{Nav-Spa14I}. This strategy ultimately lead to the solution of Brauer's conjecture \cite{Ruh22AM} and \cite{MNSFT}. For other applications of isomorphisms of character triples see \cite{Tur17}, \cite{Nav-Spa-Val}, \cite{Ros22}, \cite{Ruh22} \cite{Ros-iMcK} and \cite[Proposition 1.1]{Mar-Ros}.

In our next result, we exploit this idea in order to lift the bijections given by Proposition \ref{prop:From rational components to K_0} to the Levi subgroup $\K$. Consequently, we extend the parametrisation of unipotent $e$-Harish-Chandra series given by Theorem \ref{thm:Main Parametrisation for simple groups} for the simple group $\G$ to a parametrisation of $e$-Harish-Chandra series associated to $(\K,F)$-pseudo-unipotent characters for every $F$-stable Levi subgroup $\K$ of $\G$. First, we need a preliminary lemma.

\begin{lem}
\label{lem:Defect groups}
Let $(\L,\lambda)$ be a unipotent $e$-cuspidal pair of $(\K,F)$ and define $X:=(\G^F\rtimes \aut_\mathbb{F}(\G^F))_\K$. If $\K^F\leq H\leq \n_\G(\L)^F$ and $Q$ is an $\ell$-radical subgroup of $\n_H(\L)$, then $\c_X(Q)\leq \n_X(\L)$.
\end{lem}

\begin{proof}
Let $E:=\z(\L)_\ell^F$ and observe that $\L=\c_{\G}^\circ(E)$ according to \cite[Proposition 3.3 (ii)]{Cab-Eng94}. Now, since $\O_\ell(\n_H(\L))$ is the smallest $\ell$-radical subgroup of $\n_H(\L)$ \cite[Proposition 1.4]{Dad92}, we deduce that $E\leq \O_\ell(\n_H(\L))\leq Q$ and it follows that $\c_X(Q)\leq \c_X(E)\leq \n_X(\L)$ as wanted.
\end{proof}

\begin{theo}
\label{thm:Parametrisation for reductive groups}
For every unipotent $e$-cuspidal pair $(\L,\lambda)$ of $(\K,F)$ there exists a defect preserving $\aut_\mathbb{F}(\G^F)_{\K,(\L,\lambda)}$-equivariant bijection
\[\Omega_{(\L,\lambda)}^\K:\E\left(\K^F,(\L,{{\rm ps}_\K(\lambda)})\right)\to\irr\left(\n_\K(\L)^F\enspace\middle|\enspace {{\rm ps}_\K(\lambda)}\right)\]
such that
\[\left(X_\chi,\K^F,\chi\right)\iso{\K^F}\left(\n_{X_\chi}(\L),\n_{\K}(\L)^F,\Omega_{(\L,\lambda)}^\K(\chi)\right)\]
for every $\chi\in\E(\K^F,(\L,\ps_\K(\lambda)))$ and where $X:=(\G^F\rtimes \aut_\mathbb{F}(\G^F))_\K$.
\end{theo}

\begin{proof}
Recall that $\K_0=[\K,\K]$ and define $\L_0:=\L\cap \K_0$ and $\lambda_0$ the restriction of $\lambda$ to $\L_0^F$. Observe that $(\L_0,\lambda_0)$ is a unipotent $e$-cuspidal pair of $(\K_0,F)$. Let $z\in\z(\K^*)^{F^*}$ and consider a character $\chi$ belonging to $\E(\K^F,(\L,\lambda\hat{z}_\L))$. Since the restriction of $\lambda\hat{z}_\L$ to $\L_0^F$ coincides with $\lambda_0$, \cite[Corollary 3.3.25]{Gec-Mal20} implies that $\chi$ lies above some character in $\E(\K^F_0(\L_0,\lambda_0))$. On the other hand, suppose that $\chi\in\irr(\K^F)$ lies above some $\chi_0\in\E(\K^F_0,(\L_0,\lambda_0))$. By \cite[Proposition 3.1]{Cab-Eng94} the character $\chi_0$ has an extension $\chi'\in\E(\K^F,(\L,\lambda))$ and hence, using Gallagher's theorem \cite[Corollary 6.17]{Isa76} and \cite[(8.19)]{Cab-Eng04}, we can find $z\in\z(\K^*)^{F^*}$ such that $\chi=\chi'\hat{z}_\K$. Since $\chi'\hat{z}_\K$ is a character of $\E(\K^F,(\L,\lambda\hat{z}_\L))$ according to \cite[(8.20)]{Cab-Eng04}, we conclude that
\begin{equation}
\label{eq:Parametrisation for reductive groups, 1}
\E\left(\K^F,(\L,\ps_\K(\lambda))\right)=\irr\left(\K^F\enspace\middle|\enspace \E\left(\K^F_0,(\L_0,\lambda_0)\right)\right).
\end{equation}
Next, suppose that $\psi\in\irr(\n_\K(\L)^F\mid \lambda\hat{z}_\L)$. In this case, $\psi$ lies above the restriction of $\lambda\hat{z}_\L$ to $\L_0^F$ which coincides with $\lambda_0$. In particular, there exists some $\varphi\in\irr(\n_{\K_0}(\L_0)^F\mid \lambda_0)$ such that $\psi$ lies above $\varphi$. On the other, if $\chi$ lies above such a character $\varphi\in\irr(\n_{\K_0}(\L_0)^F\mid \lambda_0)$, then it lies above $\lambda_0$ and therefore we can find $z\in\z(\K^*)^{F^*}$ such that $\psi\in\irr(\n_\K(\L)^F\mid \lambda\hat{z}_\L)$. This shows that
 \begin{equation}
\label{eq:Parametrisation for reductive groups, 2}
\irr\left(\n_\K(\L)^F\enspace\middle|\enspace \ps_\K(\lambda)\right)=\irr\left(\n_{\K}(\L)^F\enspace\middle|\enspace \irr\left(\n_{\K_0}(\L_0)^F\enspace\middle|\enspace \lambda_0\right)\right).
\end{equation}
Finally, consider the map $\Omega_{(\L_0,\lambda_0)}^{\K_0}$ given by Proposition \ref{prop:From rational components to K_0}. Then, the result follows from \eqref{eq:Parametrisation for reductive groups, 1} and \eqref{eq:Parametrisation for reductive groups, 2} by applying \cite[Proposition 6.1 and Remark 6.2]{Ros-Generalized_HC_theory_for_Dade} as explained in the proof of \cite[Corollary 6.10]{Ros-Generalized_HC_theory_for_Dade} and using the $\K^F$-block isomorphisms of character triples obtained in Corollary \ref{cor:From rational components to K_0}. Here, we consider $A:=\G^F\rtimes \aut_\mathbb{F}(\G^F)$, $A_0:=\n_A(\L)$, $K:=\K_0^F$, $K_0=\n_{\K_0}(\L)^F=\n_{\K_0}(\L_0)^F$, $G:=\G^F$, $X:=(\G^F\rtimes \aut_\mathbb{F}(\G^F))_\K$, $\mathcal{S}:=\E(\K_0^F,(\L_0,\lambda_0))$, $\mathcal{S}_0:=\irr(\n_{\K_0}(\L_0)^F\mid \lambda_0)$, $V:=(\G^F\rtimes \aut_\mathbb{F}(\G^F))_{\K,\mathcal{S}}$ and $U:=(\G^F\rtimes \aut_\mathbb{F}(\G^F))_{\K,\L,\mathcal{Y}_0}$. Observe that the condition on defect groups required by \cite[Proposition 6.1]{Ros-Generalized_HC_theory_for_Dade} is satisfied by Lemma \ref{lem:Defect groups}.
\end{proof}

\subsection{Above $e$-Harish-Chandra series}
\label{sec:above}

We now further extend Theorem \ref{thm:Main Parametrisation for simple groups} by lifting the character bijections from Theorem \ref{thm:Parametrisation for reductive groups} with respect to normal inclusions.

\begin{prop}
\label{prop:iEBC going up}
Consider the setup of Theorem \ref{thm:Parametrisation for reductive groups} and let $\K^F\leq H\leq \n_\G(\K)^F$. Then, there exists a defect preserving $\aut_\mathbb{F}(\G^F)_{H,\K,(\L,\lambda)}$-equivariant bijection
 \[\Omega_{(\L,\lambda)}^{\K,H}:\irr\left(H\hspace{1pt}\middle|\hspace{1pt} \E\left(\K^F,(\L,\ps_\K(\lambda))\right)\right)\to\irr\left(\n_H(\L)\hspace{1pt}\middle|\hspace{1pt} \ps_\K(\lambda)\right)\]
such that
\[\left(\n_X(H)_\chi,H,\chi\right)\iso{H}\left(\n_X(H,\L)_\chi,\n_H(\L),\psi\right)\]
for every $\chi\in\irr(H\mid\E(\K^F,(\L,\ps_\K(\lambda))))$ and where $X:=(\G^F\rtimes \aut_\mathbb{F}(\G^F))_\K$.
\end{prop}

\begin{proof}
We apply \cite[Proposition 6.1]{Ros-Generalized_HC_theory_for_Dade} to the bijection given by Theorem \ref{thm:Parametrisation for reductive groups}. We consider $A:=\G^F\rtimes \aut_\mathbb{F}(\G^F)$, $G:=\G^F$, $K:=\K^F$, $A_0:=\n_A(\L)$, $X:=\n_A(\K)$, $\mathcal{S}:=\E(\K^F,(\L,\ps_\K(\lambda)))$, $\mathcal{S}_0:=\irr(\n_\K(\L)^F\mid \ps_\K(\lambda))$, $U:=X_{0,\lambda}$, $V:=X_{\mathcal{S}}$ and $J:=H$. Notice that the conditions (i)-(iii) of \cite[Proposition 6.1]{Ros-Generalized_HC_theory_for_Dade} are satisfied by \cite[Theorem 3.2 (1)]{Bro-Mal-Mic93}. Furthermore, the requirements about defect groups are satisfied by Lemma \ref{lem:Defect groups}. Therefore, as explained in \cite[Proposition 6.11]{Ros-Generalized_HC_theory_for_Dade}, we obtain the claimed result by applying \cite[Proposition 6.1 and Remark 6.2]{Ros-Generalized_HC_theory_for_Dade}.
\end{proof}

Before proceeding further, we point out an interesting analogy with another important character correspondence. The Glauberman correspondence plays a fundamental role in the study of the local-global counting conjectures and lies at the heart of most reduction theorems. In its most basic form, it states that for every finite $\ell$-group $L$ acting on a finite $\ell'$-group $K$, there exists a bijection
\[f_L:\irr_L(K)\to\irr(\n_K(L))\]
between the set of $L$-invariant characters of $K$ and the characters of the normaliser $\n_K(L)$ (see, for instance, \cite[Section 2.3]{Nav18}). A very deep result due to Dade \cite{Dad80} and recently reproved by Turull \cite{Tur08I}, shows that, if $K$ and $L$ are subgroups of a finite group $G$ and $KP\leq H\leq K\n_G(L)$, then the Glauberman correspondence $f_L$ can be lifted to a character correspondence for $H$, that is, there exists a bijection
\begin{equation}
\label{eq:Above Glauberman}
f_L^H:\irr\left(H\enspace\middle|\enspace\chi\right)\to\irr\left(\n_H(L)\enspace\middle|\enspace f_L(\chi)\right)
\end{equation}
for every $\chi\in\irr_L(K)$. On the other hand, the parametrisation of unipotent $e$-Harish-Chandra series obtained by Brou\'e, Malle and Michel \cite[Theorem 3.2]{Bro-Mal-Mic93} lies at the centre of the proofs of the local-global counting conjectures for finite reductive groups. It is interesting to note that our methods yield a character bijection above $e$-Harish-Chandra series which is analogous to \eqref{eq:Above Glauberman} in the context of the Glauberman correspondence. This is an immediate consequence of Proposition \ref{prop:iEBC going up}.

\begin{cor}
\label{cor:iEBC going up}
Consider the setup of Theorem \ref{thm:Parametrisation for reductive groups} and let $\K^F\leq H\leq \n_\G(\K)^F$. Then, there exists a bijection
\[\Psi_\chi^H:\irr\left(H\enspace\middle|\enspace \chi\right)\to\irr\left(\n_H(\L)\enspace\middle|\enspace\Omega_{(\L,\lambda)}^\K(\chi)\right)\]
for every $\chi\in\E(\K^F,(\L,\ps_\K(\lambda)))$.
\end{cor}

\begin{proof}
This follows immediately from the proof of Proposition \ref{prop:iEBC going up} by following the construction made in \cite[Proposition 6.1]{Ros-Generalized_HC_theory_for_Dade}.
\end{proof}

\section{Towards Theorem \ref{thm:Main iUnipotent} and Theorem \ref{thm:Main Unipotent}}
\label{sec:Final proofs}

Finally, we apply the results obtained in the previous sections to prove Theorem \ref{thm:Main iUnipotent} which is our main result. Then, we obtain Theorem \ref{thm:Main Unipotent} as a corollary by applying the $e$-Harish-Chandra theory for unipotent characters developed by Brou\'e, Malle and Michel \cite{Bro-Mal-Mic93} and by Cabanes and Enguehard \cite{Cab-Eng94}. Before doing so, we introduce the relevant notation and prove some preliminary results.

\subsection{Preliminaries on $e$-chains}
\label{sec:e-chains}

Our first aim is to define $e$-local structures for finite reductive groups that play a role analogue to that of $\ell$-chains in the context of Dade's Conjecture and the Character Triple Conjecture. The connection between the set of $e$-chains and that of $\ell$-chains has already been studied in \cite[Section 7.2]{Ros-Generalized_HC_theory_for_Dade}. These results provide a way to obtain Dade's Conjecture and the Character Triple Conjecture as a consequence of \cite[Conjecture C and Conjecture D]{Ros-Generalized_HC_theory_for_Dade}. The possibility to use different types of chains is crucial in the study of Dade's Conjecture and has been introduced by Kn\"orr and Robinson \cite{Kno-Rob89}. Their results were insipred by previous studies conducted by many authors including Brown \cite{Bro75} and Quillen \cite{Qui78} who analised the homotopy theory of associated simplicial complexes.

%Similar homotopic theoretic considerations can be found in \cite{Ros-Homotopy} where the simplicial complex of $e$-chains is considered.

\begin{defin}
\label{def:e-chains}
We denote by $\CL_e(\G,F)$ the set of \textbf{$e$-chains} of the finite reductive group $(\G,F)$, that is, chains of the form
\[\sigma=\left\lbrace\G=\L_0>\L_1>\dots>\L_n\right\rbrace\]
where $n$ is a non-negative integer and each $\L_i$ is an $e$-split Levi subgroup of $(\G,F)$. We denote by $|\sigma|:=n$ the length of the $e$-chain $\sigma$ and by $\L(\sigma)$ its last term. Furthermore, we define $\CL_e(\G,F)_{>0}$ to be the set of $e$-chains having length strictly larger than $0$.
\end{defin}

Observe that the notion of length defined above, induces a partition of the set $\CL_e(\G,F)$ into $e$-chains of even and odd length. More precisely, we denote by $\CL_e(\G,F)_\pm$ the subset of those $e$-chains $\sigma\in\CL_e(\G,F)$ that satisfy $(-1)^{|\sigma|}=\pm 1$.

In what follows, given an $e$-chain $\sigma$ and an $e$-split Levi subgroup $\M$ of $(\L(\sigma),F)$, we denote by $\sigma+\M$ the $e$-chain obtained by adding $\M$ at the end of $\sigma$. We also allow the possibility that $\M=\L(\sigma)$, in which case we have $\sigma+\L(\sigma)=\sigma$. Vice versa, we denote by $\sigma-\L(\sigma)$ the $e$-chain obtained by removing the last term $\L(\sigma)$ from $\sigma$. In this way we obtain $(\sigma+\M)-\L(\sigma+\M)=\sigma$ where as usual $\L(\sigma+\M)$ denotes the final term of the $e$-chain $\sigma+\M$. Here, we use the convention that $\sigma_0-\L(\sigma_0)=\sigma_0=\sigma_0+\G$ where $\sigma_0=\{\G\}$ is the trivial $e$-chain.

Next, consider the action of $\G^F$ on the set of $e$-chains $\CL_e(\G,F)$ induced by conjugation: for every $g\in \G^F$ and $\sigma=\{\L_i\}_i$, we define
\[\sigma^g:=\left\lbrace\G=\L_0>\L_1^g>\dots>\L_n^g\right\rbrace.\]
It follows from this definition that the stabiliser $\G^F_\sigma$ coincides with the intersection of the normalisers $\n_\G(\L_i)^F$ for $i=1,\dots, n$. Similarly, we can define an action of $\aut_\mathbb{F}(\G^F)$ on $\CL_e(\G,F)$ and give an analogous description of the chains stabilisers $\aut_\mathbb{F}(\G^F)_\sigma$. In particular, notice that the last term of the chain satisfies $\L(\sigma)^F\unlhd \G_\sigma^F$. Using this observation, we can use the results of Section \ref{sec:above} to control the characters of $\G^F_\sigma$ that lie above pseudo-unipotent series of $\L(\sigma)$.

\begin{defin}
\label{def:e-chains character sets}
For every $e$-chain $\sigma\in\CL_e(\G,F)$ we denote by $\CP(\sigma)$ the set of unipotent $e$-cuspidal pairs $(\M,\mu)\in\CP(\L(\sigma),F)$ that satisfy $\M<\G$. Furthermore, for any such pair $(\M,\mu)\in\CP(\sigma)$, we define the character set
\begin{numcases}{\uch(\G^F_\sigma,(\M,\mu)):=}
\irr\left(\G_\sigma^F\enspace\middle|\enspace\E\left(\L(\sigma)^F,\left(\M,\ps_{\L(\sigma)}(\mu)\right)\right)\right)& $\L(\sigma)>\M$\label{case 1}
\\
\irr\left(\G_\sigma^F\enspace\middle|\enspace\E\left(\L(\sigma)^F,\left(\M,\ps_{\L(\sigma-\L(\sigma))}(\mu)\right)\right)\right)& $\L(\sigma)=\M$\label{case 2}
\end{numcases}
\end{defin}

The need to distinguish the cases \eqref{case 1} and \eqref{case 2} will become apparent in the proofs of Proposition \ref{prop:Parametrisation for e-chains stabilisers} and Theorem \ref{thm:iUnipotent} below. Observe that in the definition above, we are excluding the degenerate case where $\G=\L(\sigma)=\M$ and therefore the chain $\sigma-\L(\sigma)$ in the case \eqref{case 2} is always defined. To understand the reason why we are excluding this case, we can consider an analogy with Dade's Conjecture. For every finite group $G$, recall that $\k(G)$ denotes the number of its irreducible characters and that, for any non-negative integer $d$, the symbol $\k^d(G)$ denotes the number of those irreducible characters of $\ell$-defect $d$. The local-global counting conjectures provide a way to determine the global invariants $\k^d(G)$ in terms of $\ell$-local structures. This idea was made precise by Isaacs and Navarro \cite{Isa-Nav22}. According to their definitions, the block-free version of Dade's Conjecture can be stated by saying that the functions $\k^d$ are chain local for every $d>0$. Consequently, and because a sum of chain local functions is chain local, we deduce that the difference $\k-\k^0=\sum_{d>0}\k^d$ is a chain local function. On the other hand, using the fact that groups admitting a character of $\ell$-defect zero have trivial $\ell$-core, it is easy to see that $\k^0$ is not chain local. The exclusion of the case $\G=\L(\sigma)=\M$ can be explained by interpreting these observations in the context of unipotent characters. Recall that $\k_{\rm u}(\G^F)$ and $\k_{\rm c,u}(\G^F)$ denote the number of unipotent characters of $\G^F$ and unipotent $e$-cuspidal characters of $\G^F$ respectively. If $\ell$ does not divide the order of $\z(\G^F)$, then \cite{Cab-Eng94} implies that the unipotent $e$-cuspidal characters of $\G^F$ have defect zero. Therefore, as in the case of Dade's Conjecture, the global invariant we want to determine $e$-locally is the difference $\k_{\rm u}(\G^F)-\k_{\rm c,u}(\G^F)$. Finally, notice that $\k_{\rm c,u}(\G^F)$ is exactly the number of unipotent $e$-cuspidal pairs $(\M,\mu)$ of $\L(\sigma)$ satisfying $\G=\L(\sigma)=\M$.

In the following lemma, we show that if the set $\uch(\G_\sigma^F,(\M,\mu))$ is non-empty then $(\M,\mu)$ is uniquely defined up to $\G_\sigma^F$-conjugation.

\begin{lem}
\label{lem:conjugation problem}
Let $\sigma\in\CL_e(\G,F)$ and consider two unipotent $e$-cuspidal pairs $(\M,\mu)$ and $(\K,\kappa)$ in $\CP(\sigma)$. If the sets $\uch(\G_\sigma^F,(\M,\mu))$ and $\uch(\G_\sigma^F,(\K,\kappa))$ have non-trivial intersection, then $(\M,\mu)$ and $(\K,\kappa)$ are $\G^F_\sigma$-conjugate.
\end{lem}

\begin{proof}
Suppose that $\vartheta$ is a character belonging to $\uch(\G_\sigma^F,(\M,\mu))$ and $\uch(\G_\sigma^F,(\K,\kappa))$. If we set $\L:=\L(\sigma)$, then we can find elements $s,t\in\z(\L^*)^{F^*}$ and characters $\varphi\in\E(\L^F,(\M,\mu))$ and $\psi\in\E(\L^F,(\K,\kappa))$ such that $\vartheta$ lies above $\varphi\hat{s}_\L$ and $\psi\hat{t}_\L$. By Clifford's theorem, we deduce that $\varphi\hat{s}_\L=(\psi\hat{t}_\L)^g$ for some $g\in \G_\sigma^F$. Furthermore, since $\hat{s}$ is a linear character, we obtain that $\varphi=\psi^g(\hat{t}_\L)^g(\hat{s}_\L)^{-1}$. Since both $\varphi$ and $\psi^g$ are unipotent characters of $\L^F$, using \cite[Proposition 8.26]{Cab-Eng04} we deduce that $(\hat{t}_\L)^g(\hat{s}_\L)^{-1}=1_\L$ and therefore $\varphi=\psi^g$. But then, \cite[Theorem 3.2(1)]{Bro-Mal-Mic93} shows that $(\M,\mu)$ and $(\K,\kappa)^g$ are $\L^F$-conjugate and the result follows.
\end{proof}

Next, we describe the block theory associated to characters in the sets introduced in Definition \ref{def:e-chains character sets}.

\begin{lem}
\label{lem:blocks in chain stabilisers}
Let $\sigma\in\CL_e(\G,F)$ and consider a unipotent $e$-cuspidal pair $(\M,\mu)\in\CP(\sigma)$ and a character $\vartheta\in\uch(\G_\sigma^F,(\M,\mu))$. Then:
\begin{enumerate}
\item the block $\bl(\vartheta)$ is $\L(\sigma)^F$-regular;
\item if the character $\vartheta$ lies above a given $\varphi\hat{z}_{\L(\sigma)}\in\E(\L(\sigma)^F,(\M,\mu\hat{z}_\M))$ for some $z\in\z(\L(\sigma)^*)^{F^*}$, then we have
\[\bl(\varphi\hat{z}_{\L(\sigma)})=\bl(\mu\hat{z}_\M)^{\L(\sigma)^F}\hspace{15pt}\text{ and }\hspace{15pt}\bl(\vartheta)=\bl(\varphi\hat{z}_{\L(\sigma)})^{\G_\sigma^F}=\bl(\mu\hat{z}_\M)^{\G_\sigma^F}\]
\item the induced block $\bl(\vartheta)^{\G^F}$ is defined.
\end{enumerate}
\end{lem}

\begin{proof}
The first point follows from Lemma \ref{lem:Block induction and regularity} by choosing $\L=\L(\sigma)$ and $H=\G_\sigma^F$. Furthermore, in the case of \eqref{case 2} observe that $\L(\sigma)\leq \L(\sigma-\L(\sigma))$ and hence $\z(\L(\sigma-\L(\sigma))^*)\leq \z(\L(\sigma)^*)$. Therefore, we can always find $\varphi$ and $z$ as in the statement of (ii). Since $\varphi$ is an irreducible constituent of the virtual character $\R_\M^{\L(\sigma)}(\mu)$, it follows from \cite[Proposition 4.2]{Cab-Eng94} (whose assumptions are satisfied by \cite[Proposition 3.3 (ii)]{Cab-Eng94}) that $\bl(\varphi)=b_{\L(\sigma)^F}(\M,\mu)=\bl(\mu)^{\L(\sigma)^F}$. Then, since $\hat{z}_\M$ is the restriction of the linear character $\hat{z}_{\L(\sigma)}$ to $\M^F$, we deduce from Lemma \ref{lem:Block induction and linear characters} that
\[\bl(\varphi\hat{z}_{\L(\sigma)})=\bl(\mu\hat{z}_\M)^{\L(\sigma)^F}.\]
Now, \cite[Theorem 9.19]{Nav98} implies that
\[\bl(\vartheta)=\bl\left(\varphi\hat{z}_{\L(\sigma)}\right)^{\G_\sigma^F}\]
and the second point follows by the transitivity of block induction. Finally, set $Q:=\z(\M)^F_\ell$ and observe that $Q\c_{\G^F}(Q)=\M^F\leq \n_{\G^F}(Q)$ by \cite[Proposition 3.3(ii)]{Cab-Eng94}. Then, \cite[Theorem 4.14]{Nav98} implies that $\bl(\mu\hat{z}_\M)^{\G^F}$ is well defined and so is $\bl(\vartheta)^{\G^F}$ by (ii) and transitivity of block induction. This concludes the proof.
\end{proof}

Using the lemma above, we can now define the following character set. This yields the $e$-local object through which we can determine the number of unipotent characters in a given block of $B$ of $\G^F$ and with a given defect $d\geq 0$ (see Section \ref{sec:Counting}).

\begin{defin}
\label{def:e-chains character sets, blocks}
Let $B$ be a block of $\G^F$ and $d$ a non-negative integer. For every $e$-chain $\sigma\in\CL_e(\G,F)$ and unipotent $e$-cuspidal pair $(\L,\lambda)\in\CP(\sigma)$ we define the character set
\[\uch^d\left(B_\sigma,(\M,\mu)\right):=\left\lbrace\vartheta\in\uch\left(\G_\sigma^F,(\M,\mu)\right)\enspace\middle|\enspace d(\vartheta)=d, \bl(\vartheta)^{\G^F}=B\right\rbrace.\]
where $\bl(\vartheta)^{\G^F}$ is defined according to Lemma \ref{lem:blocks in chain stabilisers} (iii). Furthermore, we denote the cardinality of this set by
\[\k_{\rm u}^d\left(B_\sigma,(\M,\mu)\right):=\left|\uch^d\left(B_\sigma,(\M,\mu)\right)\right|.\]
\end{defin}

To conclude this section, we show that Proposition \ref{prop:iEBC going up} can be used to parametrise the character sets from Definition \ref{def:e-chains character sets, blocks}.

\begin{prop}
\label{prop:Parametrisation for e-chains stabilisers}
Let $B$ be a block of $\G$ and $d$ a non-negative integer. If $\sigma\in\CL_e(\G,F)$ and $(\M,\mu)$ is a unipotent $e$-cuspidal pair in $\CP(\sigma)$ then there exists an $\aut_\mathbb{F}(\G^F)_{B,\sigma,(\M,\mu)}$-equivariant bijection
\[\Omega_{\sigma,(\M,\mu)}^{B,d}:\uch^d\left(B_\sigma,(\M,\mu)\right)\to\uch^d\left(B_{\sigma+\M},(\M,\mu)\right)\]
such that
\[\left(X_{\sigma,\vartheta},\G^F_\sigma,\vartheta\right)\iso{\G^F_\sigma}\left(X_{\sigma+\M,\vartheta},\G^F_{\sigma+\M},\Omega_{\sigma,(\M,\mu)}^{B,d}(\vartheta)\right)\]
for every $\vartheta\in\uch^d(B_\sigma,(\M,\mu))$ and where $X:=\G^F\rtimes \aut_\mathbb{F}(\G^F)$.
\end{prop}

\begin{proof}
First, observe that if $\M$ coincides with the last term $\L(\sigma)$ of the chain $\sigma$, then we have $\sigma+\M=\sigma$ which implies $\uch^d(B_\sigma,(\M,\mu))=\uch^d(B_{\sigma+\M},(\M,\mu))$. In this case the result holds by defining $\Omega_{\sigma,(\M,\mu)}^{B,d}$ as the identity. Therefore, we can assume that $\M<\L(\sigma)$ and define $\rho:=\sigma+\M$. Now, according to \eqref{case 1} we have
\begin{equation}
\label{prop:Parametrisation for e-chains stabilisers, 1}
\uch\left(\G^F_\sigma,(\M,\mu)\right)=\irr\left(\G_\sigma^F\enspace\middle|\enspace\E\left(\L(\sigma)^F,(\M,\ps_{\L(\sigma)}(\mu))\right)\right).
\end{equation}
On the other hand, noticing that $\M$ coincides with the last term $\L(\rho)$ of the chain $\rho$ and that $\rho-\L(\rho)=\sigma$, we obtain the equality $\E(\L(\rho)^F,(\M,\ps_{\L(\rho-\L(\rho))}(\mu)))=\ps_{\L(\sigma)}(\mu)$. Then, observing that $\G_\rho^F=\n_{\G_\sigma^F}(\M)$, we can apply \eqref{case 2} to obtain the equality
\begin{equation}
\label{prop:Parametrisation for e-chains stabilisers, 2}
\uch\left(\G^F_\rho,(\M,\mu)\right)=\irr\left(\n_{\G_\sigma^F}(\M)\enspace\middle|\enspace\ps_{\L(\sigma)}(\mu))\right).
\end{equation}
Next, we apply Proposition \ref{prop:iEBC going up} by choosing the groups in that statement to be $H=\G_\sigma^F$, $\K=\L(\sigma)$ and $(\L,\lambda)=(\M,\mu)$. By \eqref{prop:Parametrisation for e-chains stabilisers, 1} and \eqref{prop:Parametrisation for e-chains stabilisers, 2}, we deduce that there exists an $\aut_\mathbb{F}(\G^F)_{\sigma,(\M,\mu)}$-equivariant bijection
\begin{equation}
\label{prop:Parametrisation for e-chains stabilisers, 3}
\Omega^{\L(\sigma),\G_\sigma^F}_{(\M,\mu)}:\uch(\G^F_\sigma,(\M,\mu))\to\uch(\G_\rho^F,(\M,\mu)).
\end{equation}
Moreover, using the $H$-block isomorphisms given by Proposition \ref{prop:iEBC going up} together with \cite[Lemma 3.8 (b)]{Spa17}, we deduce that
\begin{equation}
\label{prop:Parametrisation for e-chains stabilisers, 4}
\left(X_{\sigma,\vartheta},\G^F_\sigma,\vartheta\right)\iso{\G^F_\sigma}\left(X_{\rho,\vartheta},\G^F_{\rho},\Omega_{(\M,\mu)}^{\L(\sigma),\G_\sigma^F}(\vartheta)\right)
\end{equation}
for every $\vartheta\in\uch^d(\G_\sigma^F,(\M,\mu))$. To conclude, observe first that $\Omega^{\L(\sigma),\G_\sigma^F}_{(\M,\mu)}$ sends characters of defect $d$ to characters of defect $d$. Moreover, by the transitivity of block induction and using \eqref{prop:Parametrisation for e-chains stabilisers, 4}, we deduce that
\[\bl(\vartheta)^{\G^F}=\bl\left(\Omega^{\L(\sigma),\G_\sigma^F}_{(\M,\mu)}(\vartheta)\right)^{\G^F}.\]
This shows that the bijection from \eqref{prop:Parametrisation for e-chains stabilisers, 3} sends characters in the set $\uch^d(B_\sigma,(\M,\mu))$ to characters in the set $\uch^d(B_{\sigma+\M},(\M,\mu))$ and therefore it restricts to a bijection, denoted by $\Omega^{B,d}_{\sigma,(\M,\mu)}$, satisfying the properties required in the statement. This completes the proof.
\end{proof}

We conclude this section with a remark on the isomorphisms of character triples obtained in Proposition \ref{prop:Parametrisation for e-chains stabilisers}.

\begin{rmk}
\label{rmk:Condition on prime}
Suppose that $\ell$ does not divide $q\pm 1$ if $\G$ is of type ${\bf A}(\pm q)$. In this case, every $e$-split Levi subgroup $\L$ of $\G$ satisfies $\L=\c_\G^\circ(\z(\L)^F_\ell)$ according to \cite[Proposition 13.19]{Cab-Eng04}. This fact can be used to show that the $\G^F_\sigma$-block isomorphisms of character triples given by Proposition \ref{prop:Parametrisation for e-chains stabilisers} can be extended to $\G^F$-block isomorphisms of character triples. First, we claim that 
\begin{equation}
\label{eq:Condition on prime}
\c_{\G^FX_{\sigma,\vartheta}}(D)\leq X_{\sigma,\vartheta}
\end{equation}
for every irreducible character $\vartheta$ of $\G^F_\sigma$ and every $\ell$-radical subgroup $D$ of $\G_{\sigma+\M}^F$. Define $Q_i:=\z^\circ(\L_i)_\ell^F$ for every $e$-split Levi subgroup $\L_i$ appearing in the chain $\sigma$. Then, using the fact that $D$ is $\ell$-radical, we obtain the inclusions $Q_i\leq \O_\ell(\G_\sigma^F)\leq D$. Therefore, every element $x\in\G^FX_{\sigma,\vartheta}$ that centralises $D$ centralises also each $Q_i$ and hence normalises each $\L_i$. It follows that
\[\c_{\G^FX_{\sigma,\vartheta}}(D)\leq (\G^FX_{\sigma,\vartheta})_\sigma=X_{\sigma,\vartheta}\]
as required by \eqref{eq:Condition on prime}. We can now apply \cite[Lemma 2.11]{Ros22} to the $\G^F_\sigma$-block isomorphisms given by Proposition \ref{prop:Parametrisation for e-chains stabilisers} to show that 
\[\left(X_{\sigma,\vartheta},\G^F_\sigma,\vartheta\right)\iso{\G^F}\left(X_{\sigma+\M,\vartheta},\G^F_{\sigma+\M},\Omega_{\sigma,(\M,\mu)}^{B,d}(\vartheta)\right)\]
for every $\vartheta\in\uch^d(B_\sigma,(\M,\mu))$.
\end{rmk}

\subsection{Proof of Theorem \ref{thm:Main iUnipotent}}
\label{sec:Proof of iUnipotent}

We are finally ready to prove our main theorem which provides a bijection for unipotent characters in the spirit of the Character Triple Conjecture \cite[Conjecture 6.3]{Spa17}. In this section, we prove a slightly stronger result that provides further information on the type of $e$-chains and isomorphisms of character triples. In the following definition we introduce the analogue of the set $\C^d(B)_\pm$ considered in the Character Triple Conjecture as defined in \cite[p. 1097]{Spa17}.

\begin{defin}
\label{def:CTC set for e-structures}
Let $B$ be a block of $\G^F$ and consider a non-negative integer $d$. We define the set
\[\CL^d_{\rm u}(B)_\pm=\left\lbrace(\sigma,\M,\mu,\vartheta)\enspace\middle|\enspace\sigma\in\CL_e(\G,F)_\pm, (\M,\mu)\in\CP(\sigma),\vartheta\in\uch^d\left(B_\sigma,(\M,\mu)\right)\right\rbrace.\]
The conjugacy action of $\G^F$ induces an action of $\G^F$ on $\CL^d_{\rm u}(B)_\pm$ defined by $(\sigma,\M,\mu,\vartheta)^g:=\left(\sigma^g,\M^g,\mu^g,\vartheta^g\right)$ for every element $g\in \G^F$ and $(\sigma,\M,\mu,\vartheta)\in\CL^d_{\rm u}(B)_\pm$. We denote by $\CL^d_{\rm u}(B)_\pm/\G^F$ the corresponding set of $\G^F$-orbits of tuples. Moreover, for every such orbit $\omega$, we denote by $\omega^\bullet$ the corresponding $\G^F$-orbit of pairs $(\sigma,\vartheta)$ such that $(\sigma,\M,\mu,\vartheta)\in\omega$ for some $(\M,\mu)\in\CP(\sigma)$. In other words, if we indicate by $\overline{(\sigma,\M,\mu,\vartheta)}$ the $\G^F$-orbit of $(\sigma,\M,\mu,\vartheta)$, then $\overline{(\sigma,\M,\mu,\vartheta)}^\bullet$ is the $\G^F$-orbit of the pairs $(\sigma^g,\vartheta^g)$. 
\end{defin}

In a similar way, if $\aut_\mathbb{F}(\G^F)_B$ denotes the set of those automorphisms $\alpha\in\aut_\mathbb{F}(\G^F)$ that stabilise $B$, then we can define $(\sigma,\M,\mu,\vartheta)^\alpha:=\left(\sigma^\alpha,\M^\alpha,\mu^\alpha,\vartheta^\alpha\right)$ for every $\alpha\in\aut_\mathbb{F}(\G^F)_B$ and $(\sigma,\M,\mu,\vartheta)\in\CL_{\rm u}^d(B)$. In this way, we obtain an action of the group $\aut_\mathbb{F}(\G^F)_B$ on the set $\CL_{\rm u}^d(B)_\pm$ and on the corresponding set of orbits $\CL_{\rm u}^d(B)_\pm/\G^F$.

\begin{theo}
\label{thm:iUnipotent}
For every block $B$ of $\G^F$ and every non-negative integer $d$, there exists an $\aut_\mathbb{F}(\G^F)_B$-equivariant bijection
\[\Lambda:\CL^d_{\rm u}(B)_+/\G^F\to\CL^d_{\rm u}(B)_-/\G^F.\]
Moreover, for every $\omega\in\CL^d_{\rm u}(B)_+/\G^F$, any $(\sigma,\vartheta)\in\omega^\bullet$ and any $(\rho,\chi)\in\Lambda(\omega)^\bullet$ we have
\[|\sigma|=|\rho|\pm 1\]
and
\[\left(X_{\sigma,\vartheta},\G^F_\sigma,\vartheta\right)\iso{J}\left(X_{\rho,\chi},\G^F_\rho,\chi\right)\]
with $J=\G^F_\sigma$, if $|\sigma|=|\rho|-1$, or $J=\G_\rho^F$, if $|\sigma|=|\rho|+1$, and where $X:=\G^F\rtimes \aut_\mathbb{F}(\G^F)$.
\end{theo}

\begin{proof}
Define $A:=\aut_\mathbb{F}(\G^F)$ and observe that $X=\G^F\rtimes A$. In a first step, we construct an equivariant bijection between triples of the form $(\sigma,\M,\mu)$. More precisely, let $\mathcal{S}$ denote the set of such triples $(\sigma,\M,\mu)$ with $\sigma\in\CL_e(\G,F)$ and $(\M,\mu)\in\CP(\sigma)$. We define a map
\[\Delta:\mathcal{S}\to\mathcal{S}\]
by setting
\[\Delta\left((\sigma,\M,\mu)\right):=\begin{cases}
\left(\sigma+\M,\M,\mu\right),& \L(\sigma)>\M
\\
\left(\sigma-\M,\M,\mu\right),& \L(\sigma)=\M.
\end{cases}\]
Observe that the chain $\sigma-\M$ is always defined since $\M<\G$ by the definition of $\CP(\sigma)$. Moreover, it is clear from the definition above that the map $\Delta$ is $A$-equivariant and satisfies $\Delta^2={\rm Id}$. Therefore, observing that $|\sigma\pm\M|=|\sigma|\pm 1$, we conclude that $\Delta$ restricts to an $A$-equivariant bijection
\[\Delta:\mathcal{S}_+\to\mathcal{S}_-\]
where $\mathcal{S}_\pm$ denotes the set of those triples $(\sigma,\M,\mu)$ of $\mathcal{S}$ that satisfy $\sigma\in\CL_e(\G,F)_\pm$. Furthermore, notice once again that if $\Delta((\sigma,\M,\mu))=(\rho,\K,\kappa)$, then
\begin{equation}
\label{eq:iUnipotent 2}
|\sigma|=|\rho|\pm 1.
\end{equation}
Now, fix an $A_B$-transversal $\mathcal{T}_+$ in $\mathcal{S}_+$ and observe that the image of $\mathcal{T}_+$ under the map $\Delta$, denoted by $\mathcal{T}_-$, is an $A_B$-transversal in $\mathcal{S}$ because of the equivariance property of $\Delta$. Consider $(\sigma,\M,\mu)\in\mathcal{T}_+$ and write $\Delta((\sigma,\M,\mu))=(\rho,\M,\mu)$. In what follows, we may assume without loss of generality that $\L(\sigma)>\M$ and that $\rho=\sigma+\M$, otherwise we repeat the arguments verbatim by replacing $(\sigma,\M,\mu)$ with $(\rho,\M,\mu)$. By Proposition \ref{prop:Parametrisation for e-chains stabilisers} we obtain an $A_{B,\sigma,(\M,\mu)}$-equivariant bijection
\[\Omega_{\sigma,(\M,\mu)}^{B,d}:\uch^d\left(B_\sigma,(\M,\mu)\right)\to\uch^d\left(B_{\rho},(\M,\mu)\right)\]
such that
\begin{equation}
\label{eq:iUnipotent 3}
\left(X_{\sigma,\vartheta},\G^F_\sigma,\vartheta\right)\iso{\G^F_\sigma}\left(X_{\rho,\chi},\G^F_{\rho},\chi\right)
\end{equation}
for every $\vartheta\in\uch^d(B_\sigma,(\M,\mu))$ and where $\chi$ is the image of $\vartheta$. Consequently, if $\mathcal{U}_{+}^{(\sigma,\M,\mu)}$ is an $A_{B,(\sigma,\M,\mu)}$-transversal in the character set $\uch^d(B_\sigma,(\M,\mu))$, then its image, denoted by $\mathcal{U}_{-}^{(\rho,\M,\mu)}$, under the bijection above is an $A_{B,(\rho,\M,\mu)}$-transversal in the character set $\uch^d(B_\rho,(\M,\mu))$ because $A_{B,(\sigma,\M,\mu)}=A_{B,(\rho,\M,\mu)}$.

Now, by the discussion in the previous paragraph and using Lemma \ref{lem:conjugation problem}, we conclude that the sets of $\G^F$-orbits
\[\mathcal{L}_+:=\left\lbrace\overline{(\sigma,\M,\mu,\vartheta)}\enspace\middle|\enspace (\sigma,\M,\mu)\in\mathcal{T}_+, \vartheta\in\mathcal{U}_+^{(\sigma,\M,\mu)}\right\rbrace\]
and
\[\mathcal{L}_-:=\left\lbrace\overline{(\rho,\M,\mu,\chi)}\enspace\middle|\enspace (\rho,\M,\mu)\in\mathcal{T}_-, \chi\in\mathcal{U}_-^{(\rho,\M,\mu)}\right\rbrace\]
are $A_B$-transversals in the sets $\CL^d_{\rm u}(B)_+/\G^F$ and $\CL^d_{\rm u}(B)_-/\G^F$ respectively. Finally, we can define the bijection $\Lambda$ by setting
\[\Lambda\left(\overline{(\sigma,\M,\mu,\vartheta)}^x\right):=\left(\rho,\M,\mu,\chi\right)^x\]
for every $x\in A_B$ and every $(\sigma,\M,\mu,\vartheta)\in\mathcal{L}_+$ and $(\rho,\M,\mu,\chi)\in\mathcal{L}_-$ satisfying $\Delta(\sigma,\M,\mu)=(\rho,\M,\mu)$ and such that
\[\chi=\begin{cases}\Omega^{B,d}_{\sigma,(\M,\mu)}(\vartheta),& \rho=\sigma+\M
\\
\left(\Omega^{B,d}_{\rho,(\M,\mu)}\right)^{-1}(\vartheta),& \rho=\sigma-\M.
\end{cases}\]
Using \eqref{eq:iUnipotent 2} and \eqref{eq:iUnipotent 3} together with the definition of $\Lambda$, we conclude that the properties required in the statement are satisfied and the proof is now complete.
\end{proof}

Now, as a consequence of Theorem \ref{thm:iUnipotent} and Remark \ref{rmk:Condition on prime}, we can finally prove Theorem \ref{thm:Main iUnipotent}.

\begin{proof}[Proof of Theorem \ref{thm:Main iUnipotent}]
Assume that $\ell$ does not divide $q\pm 1$ whenever $(\G,F)$ is of type ${\bf A}(\pm q)$. Consider the bijection $\Lambda$ from Theorem \ref{thm:iUnipotent} and chose $\omega\in\CL^d_{\rm u}(B)_+/\G^F$, $(\sigma,\vartheta)\in\omega^\bullet$ and $(\rho,\chi)\in\Lambda(\omega)^\bullet$. Then, we have
\[\left(X_{\sigma,\vartheta},\G^F_\sigma,\vartheta\right)\iso{J}\left(X_{\rho,\chi},\G^F_\rho,\chi\right)\]
with $J=\G^F_\sigma$, if $|\sigma|=|\rho|-1$, or $J=\G_\rho^F$, if $|\sigma|=|\rho|+1$. In either cases, applying Remark \ref{rmk:Condition on prime}, we deduce that
\[\left(X_{\sigma,\vartheta},\G^F_\sigma,\vartheta\right)\iso{\G^F}\left(X_{\rho,\chi},\G^F_\rho,\chi\right)\]
as required by Theorem \ref{thm:Main iUnipotent}.
\end{proof}

\subsection{Proof of Theorem \ref{thm:Main Unipotent}}
\label{sec:Counting}

Our final goal is to obtain a counting argument for unipotent characters as a consequence of Theorem \ref{thm:iUnipotent}. Recall that Dade's Conjecture provides a way to determine the number of characters in a given $\ell$-block $B$ and with a given defect $d$ in terms of $\ell$-local structures. Theorem \ref{thm:Main Unipotent} provides an adaptation of this idea to the unipotent characters of finite reductive groups by means of $e$-local structures compatible with $e$-Harish-Chandra theory (see Definition \ref{def:e-chains character sets, blocks}). For every $\sigma\in\CL_e(\G,F)$ we define
\begin{equation}
\label{eq:number of local characters}
\k^d_{\rm u}(B_\sigma):=\sum\limits_{(\M,\mu)}\k_{\rm u}^d(B_\sigma,(\M,\mu))
\end{equation}
where $(\M,\mu)$ runs over a set of representatives for the action of $\G^F_\sigma$ on $\CP(\sigma)$. Moreover, recall that $\k^d_{\rm u}(B)$ and $\k^d_{\rm c, u}(B)$ denote the number of irreducible characters belonging to the block $B$ and with defect $d$ that are unipotent and unipotent $e$-cuspidal respectively.

\begin{proof}[Proof of Theorem \ref{thm:Main Unipotent}]
To start, we determine the cardinality of the sets of $\G^F$-orbits $\CL_{\rm u}^d(B)_\pm/\G^F$. By applying Lemma \ref{lem:conjugation problem}, we obtain
\begin{equation}
\label{eq:counting 1}
\left|\CL_{\rm u}^d(B)_\pm/\G^F\right|=\sum\limits_{\sigma, (\M,\mu)}\k^d_{\rm u}(B_\sigma,(\M,\mu))=\sum\limits_{\sigma}\k^d_{\rm u}(B_\sigma)
\end{equation}
where $\sigma$ runs over a set of representatives, say $\CL_\pm$, for the action of $\G^F$ on $\CL_e(\G,F)_\pm$ and $(\M,\mu)$ runs over a set of representatives for the action of $\G^F_\sigma$ on $\CP(\sigma)$. Next, we isolate the contribution given by the trivial chain $\sigma_0:=\{\G\}\in\CL_e(\G,F)_+$ to the sum in \eqref{eq:counting 1}. In this case, we have $\L(\sigma_0)=\G$ and hence $\ps_{\L(\sigma)}(\mu)=\{\mu\}$ for every $(\M,\mu)\in\CP(\sigma_0)$ because the center $\z(\G^*)^{F^*}$ is trivial under our assumptions. Consequently, using Definition \ref{def:e-chains character sets} and Definition \ref{def:e-chains character sets, blocks}, we deduce that 
\begin{align}
\k_{\rm u}^d(B_{\sigma_0})&=\sum\limits_{(\M,\mu)}\k_{\rm u}^d(B_{\sigma_0},(\M,\mu))\label{eq:counting 2}
\\
&=\sum\limits_{(\M,\mu)}\left|\irr^d(B)\cap \E(\G^F,(\M,\mu))\right|\nonumber
\\
&=\k_{\rm u}^d(B)-\k_{\rm c,u}^d(B)\nonumber
\end{align}
where the last equality holds by \cite[Theorem 3.2 (1)]{Bro-Mal-Mic93} and recalling that every pair $(\M,\mu)\in\CP(\sigma_0)$ satisfies $\M<\G=\L(\sigma_0)$. Next, Theorem \ref{thm:iUnipotent} implies that the sets $\CL_{\rm u}^d(B)_+/\G^F$ and $\CL_{\rm u}^d(B)_-/\G^F$ have the same cardinality and therefore we conclude from \eqref{eq:counting 1} and \eqref{eq:counting 2} that
\begin{equation}
\label{eq:counting 3}
\k^d_{\rm u}(B)-\k_{\rm c,u}^d(B)+\sum\limits_{\substack{\sigma\in\CL_+\\\sigma\neq \sigma_0}}\k^d_{\rm u}(B_\sigma)=\sum\limits_{\sigma\in\CL_+}\k^d_{\rm u}(B_\sigma)=\sum\limits_{\sigma\in\CL_-}\k^d_{\rm u}(B_\sigma).
\end{equation}
Finally, noticing that $(-1)^{|\sigma|+1}=\mp1$ for every $\sigma\in\CL_\pm$, we can rewrite \eqref{eq:counting 3} as
\[\k^d_{\rm u}(B)-\k_{\rm c,u}^d(B)=\sum\limits_{\sigma\in\CL_-\cup\CL_+}(-1)^{|\sigma|+1}\k^d_{\rm u}(B_\sigma)\]
which is exactly the equality in the statement of Theorem \ref{thm:Main Unipotent}.
\end{proof}

\bibliographystyle{alpha}
\bibliography{References}

\vspace{1cm}

DIPARTIMENTO DI MATEMATICA E INFORMATICA U. DINI, VIALE MORGAGNI $67/$A, FIRENZE, ITALY

\textit{Email address:} \href{mailto:damiano.rossi00@gmail.com}{damiano.rossi00@gmail.com}

%\newpage

%\section{Questions}

%Let $\G$ be simple of simply connected type, $\ell\neq 2$ and $e$ the order of $q$ modulo $\ell$. Suppose that $(\L,\lambda)$ is a unipotent $e$-cuspidal pair of $\G$.

%\begin{enumerate}
%\item If we define the semidirect product $\wt{\G}^F\rtimes \mathcal{A}$, what are the cases where the following conditions fail?
%\begin{enumerate}
%\item[(a)] $\c_{\wt{\G}^F\mathcal{A}}(\G^F)=\z(\wt{\G}^F)$; and
%\item[(b)] $\wt{\G}^F\mathcal{A}/\c_{\wt{\G}^F\mathcal{A}}(\G^F)\simeq \aut_\mathbb{F}(\G^F)$.
%\end{enumerate}

%\item Can we prove that there exists an extension $\lambda^\diamond$ of $\lambda$ to $\n_\G(\L,\lambda)^F$ that is $\aut_{\mathbb{F}}(\G^F)_{(\L,\lambda)}$-invariant?

%\item Can we show that every $\chi\in\E(\G^F,(\L,\lambda))$ extends to $\G^F\mathcal{A}_\chi$?

%\item Can we show that every $\psi\in\irr(\n_\G(\L)^F\mid \lambda)$ extends to $(\G^F\mathcal{A})_{\L,\psi}$?

%\item Can we show that every $\psi\in\irr(\n_\G(\L)^F\mid \lambda)$ extends to $\n_{\wt{\G}}(\L)^F$?

%\end{enumerate}

\end{document}